\documentclass[a4paper]{amsart}

\usepackage[all]{xy}

\usepackage{amsmath,amssymb}
\usepackage{amscd,eucal,amsthm}
\usepackage{epsfig}
\CompileMatrices

\newtheorem{theorem}{Theorem}[section]

\newtheorem{lemma}[theorem]{Lemma}
\newtheorem{proposition}[theorem]{Proposition}
\newtheorem{stat}[theorem]{Statement}
\newtheorem{corollary}[theorem]{Corollary}
\theoremstyle{definition}
\newtheorem{definition}[theorem]{Definition}

\newtheorem{remark}[theorem]{Remark}

\numberwithin{equation}{theorem}
\def\N{{\mathcal N}}
\def\H{{\mathcal H}}
\def\Hor{{\mathcal Hor}}

\def\PP{{\mathbb P}}
\def\gg{{\mathfrak g}}
\def\ll{{\mathfrak l}}
\def\mm{{\mathfrak m}}
\def\aa{{\mathfrak a}}
\def\qq{{\mathfrak q}}
\def\pp{{\mathfrak p}}
\def\bb{{\mathfrak b}}
\def\ss{{\mathfrak s}}
\def\zz{{\mathfrak z}}
\def\tm{{\mathfrak t}}
\def\hh{{\mathfrak h}}
\def\uu{{\mathfrak u}}

\def\MD{{\mathcal D}}
\def\MX{{\mathcal X}}
\def\MZ{{\mathcal Z}}

\def\codim{{\rm codim}}

\def\Im{{\rm Im}\,}
\def\Ann{{\rm Ann}\,}
\newcommand{\oo}[1]{\mathaccent"7017{#1}}
\newcommand{\X}{\oo{X}}

\newcounter{itemnumber}

\begin{document}

\sloppy

\title[On the Local structure theorem
 ] {On the Local structure theorem and equivariant geometry of cotangent vector bundles}


\author[V.~Zhgoon ]{Vladimir S. Zhgoon}
\thanks{Partially supported by RFBR grant 09-01-00287, Grant of the President of Russian Federation supporting young scientists MK-32.2011.1, and Dynasty
 Foundation fellowship}
\address{Science Research Institute of System Studies,
Department of Mathematical problems in informatics, Nakhimovskii prospect 36-1, Moscow,  Russia}
\email{zhgoon@mail.ru}

\begin{abstract} Let $G$ be a connected reductive group acting on an irreducible
normal algebraic variety $X$. We give a slightly improved version
of  local structure theorems obtained by F.Knop and D.A.Timashev that describe an
action of some parabolic subgroup of $G$ on an open subset of
$X$. We also extend various results of E.B.Vinberg and D.A.Timashev on the set of horospheres in $X$.
 We construct a family of nongeneric horospheres in $X$ and a variety $\Hor$ parameterizing
this family, such that there is
a rational $G$-equivariant symplectic covering  of cotangent vector bundles $T^*\Hor \dashrightarrow T^*X$.
As an application we get a  description of the image of the moment map of $T^*X$  obtained by F.Knop by means of geometric methods
that do not involve differential operators.
\end{abstract}

\maketitle

Let $G$ be a connected reductive group acting on an irreducible
normal algebraic variety $X$. In this paper we discuss various
results describing the action of a certain parabolic subgroup of
$G$ on an open subset of $X$. These results are usually called
``local structure theorems''. The first results of that kind were
discovered by F.Grosshans \cite{gr}, and independently by M.Brion,
D.Luna and T.Vust \cite{BLV}. We should mention that the latter theorem was
improved by F.Knop \cite{asymp}. He used it to integrate  the invariant collective motion and
to describe the closures of so-called generic  flats for the class of varieties that he called
non-degenerate. (We recall the definition later.) In \cite{tim} D.A.Timashev proves a
generalization of the local structure theorem and this allows him to
integrate the invariant collective motion (generalizing the ideas
of F.Knop \cite{asymp}   with a  weaker assumption than
non-degeneracy).
 In this paper we give an refined version of the local structure theorem
 obtained by D.A.Timashev. One of the applications of this theorem is to study the closures of
generic flats for some class of varieties, this will be published elsewhere.

The second aim of this paper is to generalize a result of E.B.Vinberg \cite{vin} who constructed a rational Galois
cover  of $T^*X$ for a quasiaffine $X$ by the cotangent bundle to the variety of generic horospheres (this results are also valid for non-degenerate varieties).  By horospheres we call the orbits of all maximal
 unipotent subgroups of $G$ in $X$. It can be observed that the set of generic horospheres (i.e. the generic orbits of maximal unipotent subgroups of G)
  can be supplied a structure of algebraic variety.
The Galois group of this rational cover is equal to the little Weyl group of the variety $X$.
It is well known that this result could not be directly generalized to arbitrary varieties since
  the set of generic horospheres is not good enough
 for this purpose, as can be seen in the case when $X$ is a flag variety.
 These results were substantially generalized by D.A.Timashev for some class of varieties which is wider than non-degenerate
 varieties and which included flag varieties (however the former class does not contain all horospherical varieties).
 In this paper we construct a family
  of degenerate horospheres and a variety $\Hor$ parameterizing them, such that there is a
 rational covering  of the cotangent vector bundles $T^*\Hor \dashrightarrow T^*X$. It is proved that the Galois
 group of this rational covering is the little Weyl group introduced by F.Knop \cite{weylgr}.

The structure of this paper is the following. Section 1 is preliminary, we recall the local structure theorem introduced by F.Knop
and its corollaries.
   In   Section 2
we construct a $Q$-equivariant mapping $\pi_D$  from an
open subset $\X\subset X$ to a generalized flag variety of a Levi
subgroup of $Q$ (here $Q$ is the common stabilizer of the divisors of $B$-semi-invariant
rational functions, which is a
parabolic subgroup of $G$). In   Section 3 we relate the fibers of the introduced mapping $\pi_D$ to a cross section introduced
  by F.Knop.  The map $\pi_D$ allows us to give a refined version of the
local structure theorem in the sense of D.A.Timashev in   Section 4.  In
Section 5 using ideas of F.Knop \cite{kn_Harish} of studing Bia\l{}ynicki-Birula cells
 we construct a foliation of nongeneric horospheres such that the $G$-translate of the
 conormal bundle to this foliation is dense in $T^*X$.
 We note that   in the situations closely related to the topic of the present paper the Bia\l{}ynicki-Birula decomposition
 was also used by D.Luna \cite{Luna} and M.Brion  \cite{Brion}.  Section 6
 is devoted to generalization of the construction of E.B.Vinberg that relates $T^*X$ and the cotangent bundle to the constructed foliation of horospheres.
 In  Section 7 we prove that the Galois group of the rational  covering $T^*\Hor \dashrightarrow T^*X$ is equal to the little Weyl group $W_X$.
 We also give an elementary description of the image of the normalized moment map, we note that our proof does not involve differential operators (cf. Knop \cite{weylgr},\cite{kn_Harish}).
  This work should be considered as a direct
continuation of \cite{asymp},\cite{vin},\cite{tim}.

{\vspace {1ex}}

The author is grateful to D.A.Timashev for  fruitful
discussions that lead to simplification of most of the proofs \footnote{The Remark \ref{Rem_tim} and the ideas of the first proof of Proposition \ref{fiber}, of the first proof of Proposition \ref{image_mu2}
  and of the second proof of Theorem \ref{stab_of_hor} are due to D.A.Timashev, who kindly proposed them after reading a preliminary version of this paper.}
  and for explaining the formula for the normalized moment map. 
  I express my gratitude to M.Brion
for useful discussions and for careful reading  of this work. I am grateful to F.Knop
for his idea of using Bia\l{}ynicki-Birula cells in  Section 5. I would  like to thank the referees for their useful comments that considerably improved the exposition of the paper.
 {\vspace {2ex}}

{\bf Notation and conventions.}

{\vspace {1ex}}

All  varieties are considered over an algebraically closed   field $\Bbb K$ of characteristic zero.
 By Gothic letters we  denote Lie algebras
corresponding to  algebraic groups, denoted by similar capital Latin letters.  Let us choose a $G$-invariant
nondegenerate quadratic Cartan-Killing form on the algebra $\gg$ as a trace form induced
 from a faithful representation of G. This form identifies $\gg$ and $\gg^*$.
Speaking of the action $G:\gg$  (resp. $G:\gg^*$) we always assume that
it is (co)adjoint.  For $\hh\subset \gg$ by $\hh^{\bot}$ we denote the annihilator of $\hh$ in $\gg^*$.

We fix a Borel subgroup $B\subset G$, and a maximal torus
$T\subset B$. Let $B^-$ be the unique Borel subgroup of $G$ such that
$B^-\cap B=T$. By $P\supset B$   we denote a parabolic subgroup.
 By $P^-\supset B^-$ we denote the  parabolic subgroup opposite to $P$.
Denote by $P_u$ (resp. $P^-_u$) the unipotent radical of $P$
(resp. $P^-$).

Let  $\Xi=\Xi(T)$ be the lattice of  characters of $T$. Consider the lattice of one parameter
subgroups $\Lambda=\Lambda(T)$. For $\lambda:\Bbb K^*\rightarrow T$ and a character $\chi\in \Xi$ we have a pairing
$\langle \lambda,\chi \rangle$ defined by the formula $\chi(\lambda(t))=t^{\langle \lambda,\chi \rangle}$ that identifies
$\Lambda$ and $\Xi^*$.
We shall denote the element in  $\Xi^*$ corresponding to $\lambda$ by the same letter, this should not lead to confusion.
We use an additive notation  for the group law in  $\Lambda$ and in $\Xi$.

 Denote by $W=N_G(T)/T$ the Weyl group of $G$.
 $\Delta  $ is the root system of the Lie algebra $\mathfrak g$ corresponding to $T$.
 $\Delta ^{+} (\Delta ^{-})  $ is the system of  positive (negative)
 roots corresponding to the Borel subalgebra  $\mathfrak b \subset \mathfrak g$.
 $\Pi$ is the system of  simple roots.
 We also have the standard decomposition  $\gg=\tm\oplus\bigoplus\limits_{\alpha \in \Delta}\gg_\alpha$
 into the root subspaces. For $\alpha \in \Delta$ let $e_\alpha\in \gg_\alpha$ be the corresponding element of a Chevalley basis,
  $\alpha^{\vee}$ be the corresponding coroot, and  $s_\alpha$  be the corresponding reflection.
$w_0\in W$ is the longest element in the Weyl group.
By $\tm^* / W$ we denote the corresponding geometric quotient of  $\tm^*$ by $W$.
Let $L$ be a Levi subgroup of a parabolic subgroup  $P\supset B$. Assume that
$L$ contains $T$. Then $B_L=L\cap B$ is a Borel subgroup
of $L$. We  denote by $\Delta_L\subset \Delta$ and by
$\Delta^+_L\subset \Delta_L$ the root subsystem corresponding to
$L$ and its set of positive roots. The subset of simple roots in $\Delta^+_L$ is denoted by $\Pi_L$.
  We denote by $C_L\subset \Xi(T)$  the dominant Weyl chamber of $L$ with respect
  to the positive root system $\Delta^+_L$. By $C_L^\circ$ we denote the interior of $C_L$.
For a parabolic subgroup $P$ (resp. $\pp$) containing $T$ (resp.$\tm$) we denote by $\Delta(P_u)$ (resp. $\Delta(\pp_u)$) the subset of roots in $\Delta$ corresponding
to the root decomposition of $\pp_u$.

Consider the simple $G$-module  $V_\chi$ with  highest weight
$\chi$ and its dual $V^*_\chi$. The highest weight vector of
$V_\chi$  is denoted by $\sigma_\chi$ and the lowest weight vector
of $V^*_\chi$ is denoted by $\sigma^*_{-\chi}$, $\langle v,
w\rangle$ is the pairing of $v\in V_\chi$ and $w \in V^*_\chi$.
By ${\rm Wt}(V_\chi)$ we denote the set of weights for the  $T$-action on $V_\chi$.

For an algebraic group $H$ by the superscript $(-)^{(H)}$ we mean $H$-semi-invariants and by $(-)^{(H)}_\chi$ we
mean $H$-semi-invariants of  weight $\chi$. 



Let $G\supset H$ be linear algebraic  groups and  $Z$ be a quasiprojective $H$-variety.  Assume that $Z$ is normal or
 the  quotient map $G\rightarrow G/H$ is locally trivial for Zariski topology.
 Then we may form   a quasi-projective $G$-variety $G*_HZ$, by considering the quotient
 of $G\times Z$ by the action of $H$: $(g,z)\mapsto (gh^{-1},hz)$. The image of a point $(g,z)$ in this quotient we shall denote by $[g*z]$.

For an algebraic group action $G$ on $X$, $\xi x$ is the velocity vector of $\xi \in \gg$ at $x\in  X$,  $\gg x$ is the tangent space to the orbit
$Gx$ in $x$ and $G_x$ is the stabilizer of $x$.
For affine  $X$ and a group $G$, in the case when  the algebra  $\Bbb K[X]^{G}$  of $G$-invariant regular functions on $X$ is finitely generated,
by $X/\! \!/G$ we denote a quotient of $X$ which is equal to ${\rm Spec}\ \Bbb K[X]^{G}$.
If the variety $X$ is smooth we can define the moment map
$\mu_X:T^*X\longrightarrow\gg^*$ (where $T^*X$ is the cotangent
bundle of $X$) by the following formula.
$$\langle\mu_X(\alpha),\xi \rangle=\langle\alpha,\xi x \rangle, \ \forall x\in X,\ \alpha\in T_x^*X, \ \xi \in \gg.$$

Let us recall that for a homogenous variety $X=G/H$ the cotangent bundle can be expressed as:
$$T^*X\cong G*_H(\gg/\hh)^*\cong G*_H\hh^{\bot}.$$
The moment map is induced by the inclusion $\hh^{\bot}\hookrightarrow \gg^*$. Its image   is equal to $G\hh^{\bot}$.

\section{ Local structure theorem}

We begin with some preliminary remarks.
Consider a normal $G$-variety $X$ and a Cartier divisor $D=\sum a_iD_i$, where $D_i$ are
$B$-stable prime Cartier divisors. Let us  call $D$ a $B$-divisor.
We denote by $P[D_i]$ the stabilizer  of   $D_i$. The
stabilizer of the $B$-divisor $D$ is defined as  the intersection of the
stabilizers of its prime components $P[D]=\bigcap_i
P[D_i]$ (it is clearly a parabolic subgroup of $G$).
As we have $P[D+D^{'}]=P[D]\cap P[D^{'}]$,  there exists a
$B$-divisor for which $P[D]$ is absolutely minimal. 
We denote this parabolic subgroup by $P(X)$.

Replacing a divisor $D$ with a sufficiently large multiple $nD$,
we may assume that $D$ is $G$-linearized (\cite{kn}), and in
particular $B$-linearized. Any two $G$-linearizations differ by a character of $G$, we choose one of them.

 To every  $B$-divisor $D=\sum a_iD_i$
   we can associate a weight in the following way. If $D$ is effective and  $G$-linearized  we have  a section $\sigma_D \in
H^0(X,\mathcal O(D))^{(B)}_\chi$ whose scheme of zeros is  $D$, then $\sigma_D$ is a
$B$-semi-invariant vector of some weight $\chi$.  Choose  integers
$n_i$  such that the divisors $n_iD_i$
 are $G$-linearized and the corresponding sections
 $\sigma_{D_i}$ have weights $\chi_{D_i}$. To a  $B$-divisor $D=\sum a_iD_i$ we  associate
the  rational weight  $\chi_D=\sum \frac{a_i}{n_i} \chi_{D_i}$.

\begin{definition}\cite{asymp} Consider a $B$-divisor $D=\sum a_iD_i$. Let $\chi_D$ be the weight of this divisor. We
call $\chi_D$   $P[D]$-regular if $\langle \chi_D,\alpha^\vee \rangle\neq 0 \
\text{for all} \ \alpha\in \Delta(P[D]_u)$. 
\end{definition}

\begin{remark} We recall that $\langle \chi_D,\alpha^\vee \rangle= 0$
for all $\alpha\in \Delta_{L[D]}$, where $L[D]\subset P[D]$ is the Levi subgroup containing $T$.
We also note that every effective divisor $D$ with stabilizer $P[D]$ is  $P[D]$-regular.
This is a standard fact of the representation theory applied to the $B$-semi-invariant section
 $\sigma_D\in H^0(X,\mathcal O(D))^{(B)}_{\chi_D}$.
\end{remark}

Let us recall from \cite{asymp} the definition of non-degenerate varieties:

\begin{definition} A  $G$-variety $X$ is called non-degenerate if there exists a rational $B$-semi-invariant
function $f_\chi\in \Bbb K(X)^{(B)}_{\chi}$ with  divisor $D=(f_\chi)$ such that 
 $\chi$ is $P(X)$-regular.
\end{definition}

Given a $B$-divisor, we can consider the following  $P[D]$-equivariant map
(\cite{asymp}):

$$\psi_D:X \setminus D \longrightarrow \mathfrak g^*: \ x\mapsto l_x,
 \ \  where \ \ l_x(\xi)=\sum \limits_i a_i\frac{\xi\sigma_i}{\sigma_i}(x).$$


By the next lemma  we may assume that
 the $B$-invariant divisor in consideration is Cartier.

\begin{lemma}\label{D_CART}(\cite [Lemma 2.2]{asymp}) Let $X$ be a normal $G$-variety and $D\subset X$ a
prime divisor. Then $D$ is a Cartier divisor outside $Y=\bigcap
\limits_{g \in G} gD$.
\end{lemma}

Later for   $D$ we shall take a  $B$-invariant and not $G$-invariant divisor. Thus $Y$ is a proper subset of $D$,
so we may shrink the variety and consider $X \setminus
Y$ instead of $X$.

Let us recall the version of the local structure theorem obtained by
Knop.

\begin{theorem}\label{lst}(\cite[Thm. 2.3, Prop. 2.4]{asymp}) Let $X$  be a normal $G$-variety
with a $B$-divisor $D$. Assume that $\chi_D$ is
$P[D]$-regular. Then:

(i) The image of $\psi_D$ is a single $P[D]$-orbit equal to $\chi_D+\pp_u$.

(ii) For some $x_0 \in X\setminus D$ let
$$
\eta_0:=\psi_D(x_0), \ \ \ \ L:=G_{\eta_0}, \ \ \ \  Z_0 :=\psi_D^{-1}(\eta_0).
$$
Then $L$ is a Levi subgroup of $P[D]$ and there is an isomorphism
$$
P[D]*_L Z_0 \longrightarrow X\setminus D
$$

(iii) Suppose that $P[D]=P(X)$. Then the
kernel $L_0$ of the action of $L$ on $Z_0$ contains the commutator
subgroup $[L,L]$.
\end{theorem}

For simplicity we denote $P(X)$  by $P$.
Let us notice that in the theorem $x_0$ can be chosen so that $L\supset T$.

 In the situation of the local structure theorem (iii), we see that the torus
$A:=L/L_0=P/L_0P_u$ is acting effectively on $Z_0$ (The group $L_0P_u$ is denoted by $P_0$). And from
$\Bbb K(X)^{(B)}=\Bbb K(Z_0)^{(B_L)}=\Bbb K(Z_0)^{(L)}$ one can 
identify $\Xi(A)$ with the group
of characters
$$\Xi(X)=\{\chi \  | \  \Bbb K(X)_\chi^{(B)}\neq 0\}.$$
We shall write $A_X$ (resp. $\aa_{X}$) if we want to  stress a dependence on the variety $X$.
Let us recall that as a corollary of the local structure theorem we get that 
 general $P_0$-orbits coincide with general $P_u$-orbits.

\section{Equivariant maps to the flag varieties}

\vspace {2ex}  To formulate a refined version  of the local
structure theorem we introduce some additional notation.

We denote by $Q$ some parabolic subgroup of $G$ containing $P$.
 Let $M$ be the Levi subgroup of $Q$ that contains the maximal torus $T$. We have a Levi
decomposition $Q=Q_u\rtimes M$. Let us assume  that $T\cap [M,M]\subset
L_0$. Later in the proof of \ref{lst_mine_version} we shall choose $Q$ to be the stabilizer
of the divisors of $B$-semi-invariant functions; then it  satisfies this
property.

It is easy to see that $Q_u\subset P_u$ and $L\subset M$. Consider
the group $M_0=[M,M]Z(L_0)$.
We also put  $Q_0=Q_u\rtimes M_0$, so we have $A\cong M/M_0\cong
Q/Q_0$.

 Let us  embed
$\mathfrak a$, the Lie algebra of $A$, into $\mathfrak l$ as the
orthocomplement to $\mathfrak l_0$. Then the group $Z_G(\mathfrak
a)$ contains $M$.  We can describe the relations among the introduced
Lie algebras by the following picture taken from \cite{tim}.

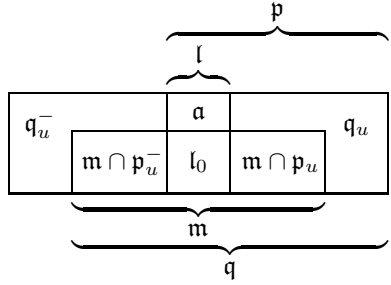
\begin{center}
\unitlength 0.55ex \linethickness{0.4pt}
\begin{picture}(60.00,43.00)
\put(30.00,18.00){\makebox(0,0)[cc]{$\strut\smash{\ll_0}$}}
\put(30.00,26.00){\makebox(0,0)[cc]{$\aa$}}
\put(43.00,18.00){\makebox(0,0)[cc]{$\strut\smash{\mm\cap\pp_u}$}}
\put(18.00,18.00){\makebox(0,0)[cc]{$\strut\smash{\mm\cap\pp^{-}_u}$}}
\put(55.00,24.00){\makebox(0,0)[cc]{$\strut\smash{\qq_u}$}}
\put(5.00,24.00){\makebox(0,0)[cc]{$\strut\smash{\qq_u^{-}}$}}
\put(10.00,13.00){\line(0,1){10.00}}
\put(10.00,23.00){\line(1,0){40.00}}
\put(50.00,23.00){\line(0,-1){10.00}}
\put(60.00,13.00){\line(-1,0){60.00}}
\put(0.00,13.00){\line(0,1){16.00}}
\put(0.00,29.00){\line(1,0){60.00}}
\put(60.00,29.00){\line(0,-1){16.00}}
\put(25.00,13.00){\line(0,1){16.00}}
\put(35.00,13.00){\line(0,1){16.00}}
\put(25.00,30.00){\makebox(0,0)[lb]{$\overbrace{\rule{5.5ex}{0pt}}^{\textstyle\ll}$}}
\put(25.00,37.00){\makebox(0,0)[lb]{$\overbrace{\rule{19.25ex}{0pt}}^{\textstyle\pp}$}}
\put(10.00,12.00){\makebox(0,0)[lt]{$\underbrace{\rule{22ex}{0pt}}_{\textstyle\mm}$}}
\put(10.00,6.00){\makebox(0,0)[lt]{$\underbrace{\rule{27.5ex}{0pt}}_{\textstyle\qq}$}}
\end{picture}
\end{center}

Identifying $\gg \cong \gg^*$ via the invariant bilinear form fixed in the conventions we see that
$\ll,\ll_0,\mm,\mm_0,\aa$ are self-dual. Also we have
$\pp_u=\pp^{\bot}\cong (\gg/\pp)^*\cong (\pp_u^-)^*$, and
$\qq_u=\qq^{\bot}\cong (\gg/\qq)^*\cong (\qq_u^-)^*$.
Let us denote
$$\aa^{pr}:=\{\xi\in \aa|\ Z_G(\xi)=Z_G(\aa), \ g\xi\notin \aa  \text{ for all} \ g \in G\setminus N_G(\aa)  \},$$
in fact $\aa^{pr}$ is obtained from $\aa$ by throwing away a finite union of hyperplanes.

We shall construct the morphism that is the main tool in the proof of the refined
local structure theorem in the sense of Timashev.

Let us fix  an effective $G$-linearized $B$-divisor $D$ and the corresponding section $\sigma_{D}$ with  weight  $\chi_D$. Consider the action of $M$
on the space of sections $H^0(X,\mathcal O(D))$. Let $V_{\chi}(M):=\langle M
\sigma_D \rangle $  be the  $M$-module generated by
$\sigma_D$. It is simple
since $\sigma_D$ is $B$-semi-invariant. $V_{\chi}(M)$ can be considered  as a simple $Q$-module fixed pointwise by  $Q_u$.
Indeed $Q_u$ is a normal subgroup in $Q$ that stabilizes  the vector $\sigma_D$. Moreover $Z(M)$ acts
by a character  on  the simple module $V_{\chi}(M)$.


Let $|V_{\chi}(M)|$ be the linear system on $X$ (possibly not
complete) corresponding to the $Q$-module $V_{\chi}(M) \subset H^0(X,\mathcal O(D))$.

\begin{remark} The basepoint set  of the linear system $|V_{\chi}(M)|$ is $M$-invariant and is equal to
$\bigcap_{m \in M} mD$ since  $V_{\chi}(M)$ is the linear span of  the $M$-orbit of $\sigma_D$.
\end{remark}

Consider the  morphism:
$$
\pi_D:X\setminus \bigcap \limits_{m \in M} mD \longrightarrow \PP(V_{\chi}(M)^*),
$$
defined by the linear system $|V_{\chi}(M)|$. 
It is easy to see that $\pi_D$ is  $Q$-equivariant. This imply the following lemma:

\begin{lemma}\label{fibers_for_QU}  Any   orbit of the radical $Z(M)  \ltimes Q_u $ of $Q$ is contained in a fiber  of $\pi_D$.
\end{lemma}


\begin{remark} If the divisor $D$ is $M$-invariant, then $\PP(V_{\chi}(M)^*)$ is a point.
For our purposes it is  sufficient to consider
a divisor that is not $M$-stable. This condition implies
that $\codim \bigcap_{m \in M} mD\geqslant 2$.
\end{remark}

Now we are ready to state one of the main theorems of the paper.

\begin{theorem}\label{main thm} Let $D$ be an effective $G$-linearized $B$-divisor,  with the canonical
  section $\sigma_{D}\in H^0(X,\mathcal O(D))^{(B)}_{\chi}$ and $P[D]$ be the stabilizer of $D$.
  Consider  a parabolic subgroup $Q$ of $G$ containing $P$ with  a Levi subgroup $M$, that satisfies
  the inclusion $T\cap [M,M]\subset
L_0$.
Then the image of the morphism:
$$
\pi_D: \X=X\setminus   ( \bigcap \limits_{m \in M} mD){\longrightarrow}
\PP (V_{\chi}(M)^*)
$$
 coincides with  the flag variety
$M\langle \sigma^*_{-\chi} \rangle\cong M/P_M^-$, where $P_M^-=M \cap P[D]^-$.    
\end{theorem}
\begin{proof}
First let us recall that $M\langle \sigma^*_{-\chi} \rangle$ is the unique closed orbit in $\PP (V_{\chi}(M)^*)$.
We begin  with the following well known lemmas.

\begin{lemma}\label{limit lemma} Let $\lambda \in C^\circ_M$. 
Denote by $\Ann \sigma_{\chi}$  a hyperplane in $V_{\chi}(M)^*$
annihilating $\sigma_\chi$.
 Then for any point  $x\in \PP (V_{\chi}(M)^*)\setminus \PP(\Ann \sigma_{\chi}) $
we have $\lim \limits_{t\rightarrow 0}\lambda(t)x=\langle \sigma^*_{-\chi}
\rangle,$
besides $Mx\nsubseteq \PP (\Ann \sigma_{\chi})$ for any $x\in  \PP (V_{\chi}(M)^*).$
\end{lemma}
\begin{proof} Let $v\in V_{\chi}(M)^*$ be a vector representing $x$.  Then $\langle Mv \rangle=V_{\chi}(M)^*$
by the simplicity of $V_\chi(M)^*$, in particular $Mx\nsubseteq \PP (\Ann \sigma_{\chi})$.
For $v\in V_{\chi}(M)^*\setminus \Ann \sigma_{\chi}$ consider the decomposition in  weight vectors:
$$
v=c\sigma^*_{-\chi}+\sum \limits_{\stackrel{\omega \in {\rm Wt}(V_{\chi}(M)^*)}{\omega\neq -\chi}} c_{\omega}\sigma^{*}_{\omega},
$$
where $\sigma^{*}_{\omega}\in V_{\chi}(M)^*$ is a vector of weight $\omega$ and $c\neq 0$.
Since  $-\langle \lambda;\chi \rangle<\langle
\lambda;\omega \rangle$, for $\lambda  \in C_M^0$ and every $\omega \in {\rm Wt}(V_{\chi}(M)^*)$
which is distinct from $-\chi$, we obtain:
$$
\lim \limits_{t\rightarrow 0}\lambda(t)x= \lim
\limits_{t\rightarrow 0}\left\langle \sigma^*_{-\chi}+\sum  \limits_{\stackrel{\omega \in
{\rm Wt}(V_{\chi}(M)^*)}{\omega\neq -\chi}} t^{\langle \lambda,\chi+\omega
\rangle} \frac{c_{\omega}}{c}\sigma^{*}_{\omega}\right\rangle =\langle
\sigma^*_{-\chi} \rangle.
$$
\end{proof}

Let us recall a following lemma.

\begin{lemma}\label{limpar}\cite[Prop. 8.4.5]{SPR} Let $G$ be a reductive group,  $T$ be a maximal torus and  $P$
be some   parabolic subgroup  containing  $T$.
Consider a one-parameter subgroup $\lambda\in \Lambda(T)$ such that $\langle \lambda,\alpha \rangle>0$ for all $\alpha \in \Delta_{P_u}$.
 Then  we have $\lim \limits_{t\rightarrow 0}\lambda(t)p_u\lambda(t)^{-1}=e$ for $p_u\in P_u$.
\end{lemma}

We note that by definition of $\pi_D$  the image $\pi_D(X\setminus D)$ is equal to the complement
of $\Ann \sigma_{\chi}$ in $\pi_D(\X)$. We also have $X\setminus D=\pi_D^{-1}(\Bbb P(V_{\chi}(M)^*)\setminus \PP(\Ann \sigma_{\chi}))$.  According to the Local
Structure Theorem \ref{lst} we may choose a dense  open $B$-invariant subset $X^\circ$
of $X\setminus D$ isomorphic to $P*_L Z_0$, where $L_0$ is acting trivially on $Z_0$.
 Let us choose a one-parameter subgroup $\lambda$ of the torus $[M,M]\cap T$
 such that $\lambda \in C^\circ_M$. Since  $[M,M]\cap T\subset L_0$, the one-parameter subgroup $\lambda$
 acts trivially on $Z_0$. Let us calculate the limit $\lim \limits_{t\rightarrow 0}\lambda(t)x$ for $x\in\pi_D(X^\circ)$ in two different  ways.
Since $\pi_D(X^\circ)\subset \PP (V_{\chi}(M)^*)\setminus \PP(\Ann \sigma_{\chi})$ Lemma \ref{limit lemma} implies the following:

\begin{stat}\label{first_limit} For $x\in \pi_D(X^\circ)$ we have $\lim \limits_{t\rightarrow 0}\lambda(t)x=\langle \sigma^*_{-\chi}
\rangle.$
\end{stat}

Let us calculate this limit in a different way.
By the Local Structure Theorem the action of $P_u$ is  free
on $X^\circ$ and we have $X^\circ=P_uZ_0$.

\begin{stat}\label{second_limit} Let $\lambda$ be a one-parameter subgroup of the torus $[M,M]\cap T$
 such that $\lambda \in C^\circ_M$. Then for $x=\pi_D(p_uz)\in \pi_D(X^\circ)$, where $p_u\in P_u$, $z\in Z_0$, we have
$\lim \limits_{t\rightarrow 0}\lambda(t)x=\pi_D(z).$
\end{stat}
\begin{proof}
Consider the decomposition $Q=Q_u\rtimes M$; combining it with the inclusions $Q_u\subset P_u \subset Q$ we get  $P_u=
Q_u\rtimes(M\cap P_u)$. Thus we have $p_u=q_um$, for $q_u\in Q_u$, $m\in M\cap P_u$.
Thus we obtain $$x=\pi_D(p_uz)=\pi_D(q_umz)=m\pi_D(z).$$

 Since $\lambda \in C^\circ_M$  is positive on all roots corresponding to $M\cap P_u$,
 by Lemma  \ref{limpar}  we get $\lim \limits_{t\rightarrow 0}\lambda(t)m\lambda(t)^{-1}=e$ for $m\in M\cap P_u$.
The triviality of the action of $\lambda$  on $Z_0$ implies that:
$$
\lim \limits_{t\rightarrow 0}\lambda(t)x=(\lim \limits_{t\rightarrow 0}\lambda(t)m\lambda(t)^{-1})\pi_D(z)=\pi_D(z)\in \pi_D(Z_0).
$$
\end{proof}

Combining  Statements \ref{first_limit} and \ref{second_limit} we get $\pi_D(Z_0)=\langle \sigma^*_{-\chi}
\rangle.$ Using the fact that $\pi_D(X^\circ)=(M\cap P_u)\pi_D(Z_0)$, we obtain that
$$M\pi_D(X^\circ)=M\pi_D(Z_0)=M\langle \sigma^*_{-\chi}
\rangle.$$
Since $M\pi_D(X^\circ)$ is dense in $\pi_D(\X)$, this proves our theorem.
\end{proof}

\begin{corollary}{\label{CorLST}}(of the proof of  Theorem 1.9) Let $X^\circ$ be the dense  open $B$-invariant subset
of $X\setminus D$ isomorphic to $P*_L Z_0$, where $L_0$ is acting trivially on $Z_0$.
 The section $Z_0$ is contained in a fiber of the map $\pi_D$.
\end{corollary}

  Now we are able to describe the sets  $X\setminus D$ and $D$ as  preimages for the map  $\pi_D$ of some subsets in $M/P_M^-$.

\begin{proposition}\label{open_cell} The set $X\setminus  D$ is
the preimage under $\pi_D$ of the open Bruhat cell $B_MP^-_M/P^-_M \subset M/P^-_{M}$, where $B_M=B\cap M$.
The set $\pi_D(D\cap \X)$ is equal to the  preimage of the complement of this open cell in  $M/P_M^-$.
\end{proposition}
\begin{proof}
Let us recall that the complement of the open cell in $M\langle \sigma^*_{-\chi} \rangle\subset \PP (V_{\chi}(M)^*)$
can be described as
$$\{\langle \sigma^* \rangle \in M/P^-_M\ | \ \langle
\sigma^* , \sigma_{\chi} \rangle=0 \} \eqno{(*)}$$

We see that equality   $\sigma_{\chi}(x)=0$ (i.e. $x \in D$) is
equivalent to $\langle \pi_D(x) , \sigma_{\chi} \rangle=0$. Thus we get
that $X\setminus D$ maps into the open cell $B_MP^-_M/P^-_M \subset M/P^-_{M}$ and $D\cap \X$
maps to the complement of the open cell. Since $\pi_D:\X \rightarrow M/P^-_M$ is surjective
this proves our proposition.
\end{proof}



\section{Relation between $\pi_D$ and cross sections.}

Now we state a result which relates the cross sections from the local structure theorem,
introduced by Knop, and the fibers of $\pi_D$.

\begin{proposition} \label{fiber} Consider an effective $G$-linearized $B$-divisor $D$ with   weight $\chi$
 and  the canonical section $\sigma_\chi$.
 We have a  map $\psi_D:X \setminus D \longrightarrow \mathfrak g^*: \ x\mapsto l_x,
 \ \  where \ \ l_x(\xi)=\frac{\xi\sigma_\chi}{\sigma_\chi}(x)$.
  For  a point $x_0 \in X\setminus D$ consider $Z
:=\psi_D^{-1}(\psi_D(x_0))$.
 Let $\pi_D$ be the map constructed
in Theorem \ref{main thm} for the divisor $D$. Then
$$Z \subset \pi_D^{-1}(\pi_D(x_0))$$
\end{proposition}
\begin{remark} This proposition is also valid for  $B$-divisor $D$ with a $P[D]$-regular weight $\chi_D$.
We omit a precise formulation and a proof for brevity.
\end{remark}
\begin{proof}
We shall give two proofs of this proposition based on different observations.

\vspace{1ex}

{\bf Argument 1}
Denote by $L[D]$ the Levi subgroup of $P[D]$ that is the stabilizer of $\psi_D(x_0)$.
By $Q$-equivariance of $\pi_D$ the translation of $x_0$ by an element from $P[D]_u$ preserves the condition $Z \subset \pi_D^{-1}(\pi_D(x_0))$;
thus we can assume that $L[D]\supset L$.
 By the local structure theorem we have
$X\setminus D=P[D]*_{L[D]}Z$. We can apply the local structure theorem to the action $L[D]:Z$ and to the
divisor $D_0\cap Z$ where $D_0\subset X$ is a $B$-divisor with  stabilizer $P:=P(X)$.
Thus we get that $Z$ contains an open subset isomorphic to $(L[D]\cap P)*_L Z_0$ where the action $L_0:Z_0$
is trivial. By Corollary \ref{CorLST} we know that $\pi_D(Z_0)=e(M\cap P[D]^-)$. Let us prove that the group $L[D]\cap P_u$ fixes this point.
Indeed from the decomposition $P_u=(P_u\cap M)Q_u$
using the root decompositions with respect to $T$ of the corresponding Lie algebras we get
 $L[D]\cap P_u=(L[D]\cap M\cap P_u)(L[D]\cap Q_u)$.
The claim follows from the facts that $Q_u$-orbits lie in the fibers of $\pi_D$ and that $L[D]\cap M$ fixes $e(M\cap P[D]^-)$.
This implies that $\pi_D((L[D]\cap P_u)Z_0)=e(M\cap P[D]^-)$ and proves the proposition.

\vspace{1ex}

{\bf Argument 2}
The set $\pi_D(X\setminus D)$ is equal to  the open $P[D]_u\cap M$-orbit in $M/(M\cap P[D]^-)$.
  If the statement of the proposition
 was not true,  there would exist a point $x\in Z$
such that $\pi_D(x)\neq \pi_D(x_0)$ and $\pi_D(x)= p\pi_D(x_0)$ for some $p\in P[D]_u\cap M$.

Consider the restriction map $$\psi_D|_{\mm}:X \setminus D
\longrightarrow \gg^*\longrightarrow\mm^*,$$
$$\psi_D|_{\mm}: x\mapsto l_x,
 \ \  where \ \ l_x(\xi)=\frac{\xi\sigma_\chi}{\sigma_\chi}(x), \ \forall \, \xi \in \mm.$$
The map $\psi_D$ is $P[D]$-equivariant, and this
implies  $P[D]\cap M$-equivariance of $\psi_D|_{\mm}$.
We have the following evident lemma.

\begin{lemma}\label{image_phi} The map $\psi_D|_{\mm}:X\setminus D\rightarrow \mm^*$ is equal to the composition of the map
$\pi_D:\X \rightarrow \PP(V_{\chi}(M)^*)$ and the map $\PP(V_{\chi}(M)^*)\setminus  \PP(\Ann \sigma_{\chi})\rightarrow \mm^*$ 
  defined as:
$$\langle \sigma^*\rangle \mapsto l_{\langle \sigma^*\rangle},
 \ \  where \ \ l_{\sigma^*}(\xi)=\frac{\langle \xi\sigma_\chi,\sigma^*\rangle }{\langle \sigma_\chi,\sigma^*\rangle}, \ \forall \, \xi \in \mm .$$
\end{lemma}


 Applying Theorem \ref{lst} we obtain that the image of $\psi_D|_{\mm}(X\setminus D)$ is a single $P[D]_u\cap M$-orbit and
 the stabilizer of   $\psi_D|_{\mm}(x_0)$ in $P[D]_u\cap M$  is trivial.
By Lemma \ref{image_phi} the equality $\pi_D(x)=p\pi_D(x_0)=\pi_D(px_0)$ for $p\in P_u\cap M$
implies that
 $\psi_D|_{\mm}(x)=\psi_D|_{\mm}(px_0).$ Taking into
account that
$$
\psi_D|_{\mm}(x_0)=\psi_D|_{\mm}(x)= \psi_D|_{\mm}(px_0)=p\psi_D|_{\mm}(x_0),
$$
we come to a contradiction since $p$ does not stabilize
$\psi_D|_{\mm}(x_0)$.
\end{proof}


\begin{proposition}\label{irr_fibers} The fibers of the map $\pi_D$ are irreducible.
The set $\pi_D^{-1}(\pi_D(x_0))$ is identified with $$(L[D]\cap Q)\ltimes (P[D]_u\cap  Q_u)*_{L[D]\cap Q}Z
\cong (P[D]_u\cap  Q_u)\times Z.$$\end{proposition}
\begin{proof}
Let $x\in M/P^-_M$ be the point corresponding to the right coset $eP^-_M$.
The preimage $\pi_D^{-1}((P[D]_u\cap M)x)$ of the open cell in $M/P^-_M$ is irreducible, being
the dense open subset $X\setminus D$ in $X$. It is isomorphic to $P[D]_u\cap M\times \pi_D^{-1}(x)$,
which implies the  irreducibility  of  $\pi_D^{-1}(x)$.  Since the fibers of $\pi_D$ are permuted transitively via the action of $M$
this proves the first part of the  proposition. The second part follows immediately from the isomorphism
$X\setminus D\cong P[D]_u\cap M\times P[D]_u\cap Q_u \times Z$, inclusion $Z\subset \pi_D^{-1}(\pi_D(x_0))$ and the freeness of the
action of $P[D]_u\cap M$ on $M/P^-_M$.
\end{proof}

Consider the set of fundamental weights $\{\omega_\alpha \}$ of $[M,M]$ where  $\alpha\in\Pi_M$.

\begin{proposition}{\label{linear_comb_of_shub_cells}} Let $\alpha$ be a simple root in $\Pi_M\setminus \Pi_L$ and  $\omega_\alpha$
be the corresponding  fundamental weight of $M$.  Then there exists a unique prime   $B$-divisor $D_\alpha\subset X$
 such that the restriction of the weight  $\chi_{D_\alpha}$ to ${[M,M]\cap T}$ is equal to $\omega_\alpha$.  Let   $D$ be a $B$-invariant divisor
 of weight $\chi_D$ such that $\chi_D|_{[M,M]\cap T}=\sum_{\alpha \in \Pi_M\setminus \Pi_L} n_\alpha\omega_\alpha$.
 Then $D=\sum_{\alpha \in \Pi_M\setminus \Pi_L} n_\alpha D_\alpha+D_0$
 for some $Q$-invariant divisor $D_0$.
\end{proposition}
\begin{proof}
Consider an effective $B$-invariant divisor $D$ (see Lemma \ref{D_CART}) such that $P[D]=P$ and let us denote $P_M=P\cap M$.
For a divisor $F_\alpha=\overline{B_Ms_{\alpha}P_M^-/P_M^-}$ (with the stabilizer  $P_\alpha$ in $M$) let us take a natural number $n$ such that $nF_\alpha$
is $M$-linearized.
 By Proposition \ref{irr_fibers} the $B$-divisor $D_\alpha:=\pi_D^{-1}(F_\alpha)$ is irreducible.
  $nD_\alpha$ is the zero divisor of a section of the
$M$-linearized line bundle $\pi_D^*(\mathcal O(nF_\alpha))$ which is  
 a pullback  the $B\cap M$-semi-invariant section $\sigma_{nF_\alpha}$.  We also have $\chi_{D_\alpha}|_{[M,M]\cap T}=\omega_\alpha$.

Let $D$ be a prime $B$-invariant divisor that is not $M$-invariant.  Since $\pi_D(D)$  is prime, by Proposition \ref{open_cell},
it is equal to a prime Schubert divisor corresponding to some simple root $\alpha \in \Pi_M\setminus \Pi_L$.
Thus $\chi_D|_{[M,M]\cap T}$ is equal to the restriction of the weight of this Schubert divisor which is $\omega_\alpha$.

 To prove the uniqueness of  $D_\alpha$ assume that we have two effective divisors $D_\alpha$ and $D_\alpha'$
 such that the restrictions to $[M,M]\cap T$  of $\chi_{D_\alpha}$ and $\chi_{D_\alpha'}$ are proportional to $\omega_\alpha$.
 This implies that the stabilizer in $M$ of $D_\alpha+D_\alpha'$ is equal to $P_\alpha$.
 Consider the map $\pi_{\alpha}:=\pi_{D_\alpha+D_\alpha'}$.
  By Theorem \ref{main thm} the image of $\pi_{\alpha}$ is equal to $M/ P^-_\alpha$  and by Proposition \ref{open_cell}
 the divisor ${D_\alpha+D_\alpha'}$ maps to the prime divisor $F_{\alpha}$
 which is the only $B_M$ invariant divisor in  $M/P^-_\alpha$.
This contradicts the irreducibility of the fibers of $\pi_\alpha$  and proves the first part of the corollary.

From the above we get a decomposition $D=\sum_{\alpha \in \Pi_M\setminus \Pi_L} m_\alpha D_\alpha+D_0$ for some $Q$-invariant $D_0$, and integer $m_\alpha$.
Comparing the $([M,M]\cap T)$-weights  of both sides of this  equality and using  the equality $\chi_{D_\alpha}|_{[M,M]\cap T}=\omega_\alpha$, we get    $\chi_D|_{[M,M]\cap T}=\sum_{\alpha \in \Pi_M\setminus \Pi_L} m_\alpha \omega_\alpha$.
From the linear independence of fundamental weights  we  have $m_\alpha=n_\alpha=\langle \chi_D, \alpha^\vee \rangle$.
\end{proof}

\begin{corollary}{\label{rat_funct}} Let $D$ be a $B$-divisor such that its weight $\chi_D$ has zero restriction to $[M,M]\cap T$.
Then the divisor $D$ is $Q$-invariant. In particular every $B$-semi-invariant rational function on $X$ is $Q$-semi-invariant.
\end{corollary}

\section{Refined local structure theorem}

Using  Theorem \ref{main thm} we shall derive a stronger
version of the local structure theorem in the sense of Timashev
\cite[Thm.3]{tim}. Let us first  introduce some notation.

We shall choose  a  parabolic subgroup $Q\supset P$ with Levi subroup $M$, for which
$[M,M]\cap T\subset L_0$.
Consider the set $\mathcal E$ of  prime $B$-divisors that occur
in the divisors $(f)$ of rational $B$-semi-invariant functions
$f\in k(X)^{(B)}$. In other words: $$\mathcal E:=\{E|\ \exists
f\in k(X)^{(B)}, \ (f)=\sum a_DD,  \ a_E\neq 0 \}.$$

We denote by $Q\supset P$ the parabolic subgroup that is the
stabilizer of all $B$-divisors from $\mathcal E$.
Since $Q$  stabilizes the divisor $(f)$ for $f\in k(X)^{(B)}$, by  a consequence of a theorem of Rosenlicht \cite[Thm.3.1]{VP}
we obtain that $f$ is $Q$-semi-invariant. Let $M$ be the Levi
subgroup of $Q$ containing the maximal torus $T$. 
We  note that $T\cap [M,M]\subset
L_0$, since  $T\cap [M,M]$ is acting trivially on $f\in
k(X)^{(B)}$. By  Corollary \ref{rat_funct} this parabolic subgroup $Q$ is maximal in the set of parabolic subgroups
  stabilizing some $B$-divisor (in particular containing $P$) and satisfying the property  $T\cap [M,M]\subset
L_0$.


We choose an effective  $B$-divisor $D$ such that the stabilizer of $D$ in $Q$ is equal to $P$.
Let us also choose an effective  divisor $E\in \mathcal E$,  such that $Q$ is the stabilizer
of $E$ in $G$. In particular the stabilizer of $D\cup E$ is equal to $P$. Let $\sigma_E\in H^0(X,\mathcal O(E))^{(Q)}$
 be the $Q$-semi-invariant section that defines $E$.
 Since $E$ is effective the weight $\chi_E$ of the divisor $E$ is $Q$-regular.
Let us denote by  $X_1:=\X\setminus E$  an open
$Q$-invariant subset and by $X^\circ$ the $P$-invariant subset $X\setminus  (D \cup
 E)$. We recall that $\X=X\setminus \bigcap_{m\in M} mD$.
We  also consider the map:
 $$\psi_E: \ X \setminus
E\longrightarrow \gg^*: \ x \longrightarrow l_x(\xi)=\frac{\xi \sigma_E}{\sigma_E}(x).$$

\begin{theorem}\label{lst_mine_version} There exists a point  $x_0\in X^\circ$
 such that $M_{\pi_D(x_0)}=P^-_M$ and $M$ is the stabilizer of $\psi_E(x_0)$.
For the $M$-stable closed
subset $Z_1:= \psi_E^{-1}\psi_E(x_0)$ of $X_1$ we have a $Q$-equivariant
 isomorphism
$$Q*_M Z_1\longrightarrow X_1.$$
 Moreover $Z_1\cong
M/P^-_M\times Z_0$ for $Z_0 :=\pi_D^{-1}\pi_D(x_0)\cap Z_1$. Here
  $M$ acts on the product $M/P^-_M\times Z_0$ by  left multiplication on $M/P^-_M$ and
   via the quotient $M/M_0\cong A$ on $Z_0$.
\end{theorem}
\begin{proof} 
Consider a point $x_0\in X^\circ$; since the Levi subgroups of $Q$ are conjugate by the elements of $Q_u$,
 translating $x_0$ by the element of $Q_u$ we may assume that $M$ is the stabilizer of $\psi_E(x_0)$.
By Proposition \ref{open_cell}  $X^\circ$ is the preimage under $\pi_D$ of the open Bruhat cell $(P_u\cap M)P^-_M/P^-_M$ in $M/P^-_M$.
In particular translating $x_0$ by the element of $(P_u\cap M)$
we can assume that $x_0$ is  in $X^\circ$, the stabilizer of $\psi_E(x_0)$ is still $M$ and $M_{\pi_D(x_0)}=P_M^-$.

The isomorphisms $Q*_M Z_1\longrightarrow
X_1$ and $Q*_M Z_1 \cong Q_u\times Z_1$ follows from the local structure
theorem \ref{lst} applied to the point $x_0\in X^\circ$, the divisor $E$ and its
stabilizer $Q$. 
From the $M$-equivariance of $\pi_D$ and $M$-invariance of $Z_1$ we  see that $Z_1=MZ_0$.

\begin{lemma}\label{weak_local_structure}
We have the isomorphism $P*_LZ_0 \tilde{\longrightarrow }X^\circ,$
where the group $L$ acts on $Z_0$ via the quotient $L/L_0$.
\end{lemma}
\begin{proof}
Let us notice that
$
\pi_D(X^\circ)\supset \pi_D(Z_1\cap X^\circ)\supset P_M\pi_D(x_0)\cong P_M/L.
$
This implies that the variety $Z_1\cap X^\circ$ (which is equal to $Z_1\setminus D$)  projects $P_M$-equivariantly onto the open cell $\pi_D(X^\circ)=(P_u\cap M)P^-_M/P^-_M$ that is isomorphic to
  $P_M/L$, and in particular $Z_1\setminus D\cong P_M*_LZ_0$.

The action of $Q_u$ is free on $X^\circ$ and we have $X^\circ \cong Q_u\times (Z_1\setminus D)$. Since $P=Q_u\rtimes P_M$
we have: $$X^\circ\cong P*_{P_M}(Z_1\setminus D)\cong P*_LZ_0.$$

Let us apply the local structure theorem to the effective divisor $ D \cup
 E$ and the group $P$. We get that $X^\circ\cong P*_LZ'_0$ for some $Z'_0$ with  trivial action  of $L_0$.
The $L$-equivariant isomorphisms $X^\circ/P_u\cong Z'_0\cong Z_0$ imply the triviality of the action of $L_0$ on $Z_0$ which proves the lemma.
\end{proof}

We shall prove the isomorphism $Z_1\cong M/P^-_M\times Z_0$ in several steps.
 Let us first prove that the $P_u\cap M$-orbit of each point
$z\in Z_1\cap X^\circ$  is contained as a dense open subset in the
orbit $M_0z$. This was first noticed by Timashev (\cite{tim}) in less general
settings for a general point of $Z_1$.

{\bf Step 1.} Since the divisor  of any  $B$-semi-invariant rational function $f\in \Bbb K(X)^{(B)}$ is
 $Q$-semi-invariant, as  noted before this function should be $Q$-semi-invariant as well.
 Thus we get
$\Bbb K(X)^{(B)}=\Bbb K(X)^{(P)}=\Bbb K(X)^{(Q)}$. Since $\Bbb K(X)^U$ is generated by $\Bbb K(X)^{(B)}$, these equalities imply that
$\Bbb K(X)^U=\Bbb K(X)^{P_0}=\Bbb K(X)^{Q_0}$.
We shall prove that a general $Q_0$-orbit contains a general $P_0$-orbit as a dense open subset.
Suppose the contrary. Then a general $Q_0$-orbit contains an infinite family of $P_0$-orbits.
Due to the Rosenlicht Theorem  general  $P_0$-orbits are separated by rational invariants
from $k(X)^{P_0}$. This implies that there exists an invariant  $f\in k(X)^{P_0}$ which has non-constant value
on  the family of $P_0$-orbits contained in a general $Q_0$-orbit, that contradicts the inclusion $f\in k(X)^{Q_0}$.

{\bf Step 2.} The isomorphism  $Q*_M Z_1\cong Q_u\times Z_1\longrightarrow X_1$  implies
that for $z\in Z_1$ we have the following isomorphisms
$Q_0z\cong Q_u\times M_0z$  and  $P_0z\cong Q_u\times (P\cap M_0)z.$
From the above step and since $X^\circ= P_uZ_0\cong P_u\times Z_0$,   for a sufficiently general
point $z\in Z_0$ we have
$\overline{Q_0z}=\overline{P_0z}$.
Thus for a  sufficiently general point $z\in Z_0$ the previous equalities give $$\overline{(P_u\cap M)z}=\overline{M_0z}.$$

{\bf Step 3.} Let us prove that $M_0z\cong {M_0}/P_{M_0}^{-}\cong M/P_M^-$ for $z\in Z_0$, where $P_{M_0}^-:=P^-\cap M_0$.

{\bf Step 3a).} First assume that $z\in Z_0$ is  a sufficiently general point.
Let us notice that  $(P_u\cap M)z$ maps isomorphically to $(P_u\cap M)P^-_M/P^-_M$.
The $M$-equivariance of $\pi_D$ and the fact
that $\pi_D$ is an isomorphism on the  dense subset $(P_u\cap M_0)z$ in $M_0z$ imply that $\pi_D$
maps $M_0z$ isomorphically to $M/P_M^-$. In addition we have
$(M_0)_z=(M_0)_{\pi_D(z)}=P_{M_0}^-$.

{\bf Step 3b).} Consider an arbitrary point $z\in Z_0$.
 Let us recall that for a general point $z_0\in Z_0$
we have $(M_0)_{z_0}=P_{M_0}^-$. Since the set of $P_{M_0}^-$-fixed points is closed in $Z_0$ we get the
inclusion $(M_0)_z\supset P_{M_0}^-$.
 Using the $M$-equivariance of $\pi_D$  we  get the following sequence of inclusions:
$$P_{M_0}^-\subset
(M_0)_z\subset (M_0)_{\pi_D(z)}=P^-_{M_0}.$$
The right
and the left side of this sequence are equal, thus all the inclusions must be
equalities, in particular $(M_0)_z=P_{M_0}^-$ for all $z\in Z_0$.

  \vspace {2ex}

   Since $Z_1\cong M*_{P_M^-}Z_0$ and $P^-_{M_0}$ acts on $Z_0$ trivially  we  get
$Z_1\cong M/ P^{-}_M\times
Z_0$, where  $M$ acts on $M/ P^{-}_M$ by left
multiplication  and via the quotient $A=M/M_0=P_M^-/P_{M_0}^-$  on $Z_0$.
\end{proof}

 To get the variant of the local structure theorem
obtained by Timashev \cite[Thm.3]{tim}, we have to take the open subset of
$Z_0=\pi_D^{-1}\pi_D(x_0)\cap \psi_E^{-1}\psi_E(x_0)$ on which the
$A$-action is free. Taking a smaller open  subset of $Z_0$, we may suppose
that it is isomorphic to $A\times C$, where the action on $C$ is trivial.

\begin{corollary}(cf. \cite[Thm.3]{tim}) There is a locally closed subset $C$ of $X$ such that the map
$$Q*_M (M/P_M^{-}\times
A\times C)\longrightarrow X$$
is an isomorphism to an open $Q$-invariant subset of $X$.
\end{corollary}

\begin{remark}\label{Rem_tim} In  Theorem \ref{lst_mine_version} we can take any parabolic subgroup $Q$ which is the stabilizer of
some  $B$-invariant divisor $E$ and that satisfies the condition $T\cap [M,M]\subset
L_0$. Corollary \ref{rat_funct} states that such a parabolic subgroup $Q$  stabilizes  the divisors of all $B$-semiinvariant rational functions on $X$
and the proof of Theorem \ref{lst_mine_version} can be generalized to the case of such $Q$.
However instead  of Step 1 of the proof of Theorem \ref{lst_mine_version}, one can use the   argument from Timashev \cite[Claims 1,2]{tim} to prove that $\mm_0x\subseteq \pp_ux$
for general $x\in X^\circ$.
In particular that will imply that $\qq_0x=\pp_ux$, as well as that $P_ux$ is open in $Q_0x$.
We shall reproduce the argument of Timashev for convenience of the reader. \end{remark}

\begin{lemma}(cf. \cite[Claims 1,2]{tim}) $\mm_0x = (\pp_u\cap \mm) x$ for sufficiently general $x\in Z_0$.
\end{lemma}
\begin{proof} From the local structure theorem we have $X^\circ=P*_LZ_0\cong P_u\times Z_0$. Consider an open subset of
 $Z_0$, we can assume that it is isomorphic to $A\times C$, for some $C$ with trivial $T$-action.
 Without loss of generality we may assume that $x\in C$. Suppose that $\mm_0x \neq (\pp_u\cap \mm)x$.
  From the freeness of the $P_u$-action we get $\qq_0x\neq \pp_ux$, this implies  the existence of $\xi_x\in T^*_xX$
 such that $\xi_x|_{\pp_ux}=0$ and $\xi_x|_{\qq_0x}\neq 0$. In particular since $P_ux=P_0x$ the first equality implies that $\mu_X(\xi_x)\in \pp_0^\bot=\aa+\pp_u$.
  Consider the $L_0$-module isomorphism $T^*_xX\cong (\aa+\pp_u) \oplus T^*_xC$; from $\xi_x|_{\pp_ux}=0$ we also have
 $\xi_x\in \aa x \oplus T^*_xC$. In particular $\xi_x$ is $L_0$-invariant. This implies that $\mu_X(\xi_x)\in (\aa+\pp_u)^{L_0}\subset
 \aa+\pp_u^{T\cap [M,M]}\subset \aa+\qq_u$ (where the last inclusion is due to the fact  that $\pp_u\cap \mm$ is
 $T$-invariant but does not contain $T\cap [M,M]$-fixed vectors).
 Thus $\mu_X(\xi_x)\bot \qq_0$, contradicting the assumption $\xi_x|_{\qq_0x}\neq 0$.
\end{proof}

\section{Families of nongeneric horospheres}

In this section we shall construct  a family of nongeneric horospheres. By horospheres we call the orbits of maximal
 unipotent subgroups of $G$.
 It will be proved that  for the conormal
bundle $\mathcal N^*X$ to some  foliation of $\overline{U}$-orbits (for some maximal unipotent subgroup
  $\overline{U}$) constructed below, we have $\overline{G\mathcal N^* X}=T^{*}X$. The construction is based on  ideas
of F.Knop \cite{kn_Harish}.
Our main idea to construct a Bia\l{}ynicki-Birula cell by means of  special choice of a one-parameter subgroup that allows us to avoid using compactifications as in the cited paper
of F.Knop. It also provides a deeper study of the constructed conormal bundle. This section is independent of the previous ones,
so some notation will be slightly changed for brevity.

The crucial step is first to consider  the case of a horospherical variety. Let us recall that a variety is called horospherical
if the stabilizer of a general point contains a maximal unipotent subgroup.  The  case of general variety $X$ will be studied by reducing to this one
by horospherical contraction, whose definition and existence is stated  below in Proposition \ref{ContractionHor}.
In a horospherical variety $X$ one can find a $G$-invariant open subset isomorphic to $G/P^{-}_0\times C$, where  $C$ is supplied
with the trivial action of $G$ and $P^-_0=L_0P^-_u$.  Thus we shall construct a variety of degenerate horospheres for $X=G/P^{-}_0$
and extend it to a horospherical variety $X$ by taking  product with $C$.
Let us introduce additional notation $M:=Z_G(\aa)$ and $M_0:=[M,M]Z(L_0)$. We note that  $M$  is not related with the group introduced in Section 2
and this notation  retains to the end of the paper.

\begin{proposition}\label{horospherical} Let $X$ be  the horospherical variety $G/P^{-}_0$. Consider a Borel subgroup $\overline{B}\subset G$
with unipotent radical $\overline{U}$ such that there is an inclusion
$\overline{\bb}\supset \aa+(\pp^{-}_u\cap \mm)$ for its Lie algebra. Let  $\mathcal N^*$ be the conormal bundle to the 
orbit
$\overline{U}P^{-}_0/P^{-}_0$, then we have $\overline{G\mathcal N}=T^{*}X$.
\end{proposition}
\begin{proof}
Let us remind that the cotangent bundle $T^{*}X$ is identified with $G*_{P^{-}_0}\pp^{- \bot}_0\cong G*_{P^{-}_0}(\aa+\pp_u^-)$.
The fiber of conormal bundle to the orbit $\overline{U}P^{-}_0/P^{-}_0$  in the point $eP^{-}_0$
 is identified with $(\overline{\uu}+\pp_0^-)^\bot=\overline{\uu}^\bot\cap \pp_0^{-\bot}=
 \overline{\bb}\cap (\aa+\pp_u^-)\supset \aa+(\pp^-_u\cap \mm)$.

 We need the following lemma.
\begin{lemma}\label{orbit_in_IM_phi} Let $P$ be some parabolic subgroup of $G$ 
 and $L$ be a Levi subgroup of $P$. Then for  a subalgebra $\aa\subseteq \zz(\ll)$ we have an equality $\overline{P_u(\aa+(\pp_u\cap \mm))}=\aa+\pp_u$.
\end{lemma}
\begin{proof} We shall prove that the  map $P_u\times(\aa+(\pp_u\cap \mm))\rightarrow \aa+\pp_u$ is dominant if we prove that its differential is surjective at the point $(e,\xi)$, 
for general $\xi \in \aa$.  Calculating the differential  in $(e,\xi)$  and using the equality $\pp_u=(\pp_u\cap\zz_\gg(\xi)) \oplus [\pp_u,\xi]$ for any $\xi\in \aa$ and
the equality $\pp_u\cap \zz_\gg(\xi) =\pp_u\cap \mm$ which holds for general $\xi \in \aa$,  we obtain
that the differential:
$$\pp_u\times (\aa+(\pp_u\cap \mm))\rightarrow [\pp_u,\xi]+\aa+(\pp_u\cap \mm)=\aa+\pp_u$$
is indeed surjective.
\end{proof}

The group $P_u^-$ is acting on the fiber of $T^*X$ over  $x_0=eP^{-}_0$
since it lies  in stabilizer of this point. From  the preceding lemma
we obtain that $$\overline{P^-_u\mathcal N^*_{x_0}}\supset \overline{P^-_u(\aa+(\pp^{-}_u\cap \mm))}=\aa+\pp^-_u=T^*_{x_0}X.$$
This implies that $$\overline{G\mathcal N^*}=G(T^*_{x_0}X)=T^*X.$$
\end{proof}

We shall need the following elementary lemma.

\begin{lemma}\label{submersion} Let $X/T$ and $Y/T$ be two families of equidimensional varieties over some variety $T$ with a $T$-morphism
$f:X/T\rightarrow Y/T$. Suppose  there exists  smooth points $t_0\in T$ and $x_0\in X_{t_0}$ such that the varieties $X,\ X_{t_0}$ are smooth at $x_0\in X_{t_0}$; $Y, \ Y_{t_0}$ are smooth at $f(x_0)$, the map
$f_{t_0}:X_{t_0}\rightarrow Y_{t_0}$ is a submersion at the point $x_{0}$ (i.e the map of tangent
spaces $df_{t_0}:T_{x_0}X_{t_0}\rightarrow T_{f(x_0)}Y_{t_0}$ is surjective) and the projections $X\rightarrow T$, $Y\rightarrow T$
are submersions at $x_0$ and $f_{t_0}(x_0)$. Then the morphism
$f_t:X_t\rightarrow Y_t$ is a submersion (and in particular is dominant) at a general point of $X_t$ for a sufficiently general $t$.
\end{lemma}

Before stating one of the main theorems of this section we recall the following proposition about adjoint orbits. Let us temporally change our notations
 only for this proposition.
\begin{proposition}\label{nilp_orbits}\cite[\S 5.1, 5.5]{ColMC} Consider an arbitrary parabolic subgroup $P$ in $G$, a Levi subgroup $L$ and the unipotent radical $P_u$.
Let $\mathcal O_{\ll}$ be a nilpotent adjoint orbit of $L$ in $\ll$. Let $x\in \zz(\ll)$ be an arbitrary element of the center of $\ll$.
There exists a unique $G$-orbit $\mathcal O_{\gg}$ meeting $x+\mathcal O_{\ll}+\pp_u$ in a dense open subset. The intersection
$\mathcal O_{\gg}\cap (x+\mathcal O_{\ll}+\pp_u)$ is a single $P$-orbit. The  following equality  holds $\codim_{\gg} \mathcal O_{\gg}=
 \codim_{\ll} \mathcal O_{\ll}$. For a general point $z\in x+\mathcal O_{\ll}+\pp_u$  the stabilizer of $z$ in $\pp^-_u$ is trivial and $[\pp^-_u,z]$
is transversal to $[\ll, z]+\pp_u$. We also have the equality for irreducible components $(P_z)^0=(G_z)^0$.
\end{proposition}

Now we are ready to deal with the case when  $X$ is an arbitrary $G$-variety.
We are going to construct a family of $\overline{U}$-orbits for some maximal unipotent subgroup $\overline{U}$,
such that the $G$-translate of  the conormal bundle to this foliation is dense in $T^*X$.

\begin{theorem}\label{degenerate_horospheres} Let $X$ be a smooth  $G$-variety. Consider the open subset $X^\circ\cong P*_LZ$ obtained by application of the Local Structure
Theorem \ref{lst}
to some effective $B$-divisor with  stabilizer equal to the parabolic
subgroup $P:=P(X)$. Then there exists a maximal unipotent subgroup $\overline{U}$ with the following properties.
\begin{itemize}
\item[](i) For any $z\in Z$ we have $\overline{U}z=(\overline{U}\cap U) z$.
\item[](ii) Let $\mathcal N^*X$ be the conormal bundle to the foliation  of orbits $\overline{U}z$ for
$z\in Z$. Then we have $\overline{G\mathcal N^*X}=T^{*}X$.
\end{itemize}
\end{theorem}
\begin{proof} To construct the desired family we proceed in several steps.

{\bf Step 1.} Our aim is to construct a Bia\l{}ynicki-Birula cell with respect to a one parameter subgroup $\lambda(t)\subset Z(L_0)$.
 We shall choose $\lambda$ in a special way. 
Let us recall that  $ \pp\cap \mm_0$ is a parabolic subalgebra of $\mm_0$ with a Levi subgroup $\ll_0$ and  the unipotent radical $\pp_u\cap \mm$.

\vspace{1ex} {\it
Let us take a one-parameter subgroup $\lambda:\Bbb K^\times \rightarrow T$ such that $\lambda(t)\in Z_{M_0}(L_0)$
 and $\langle \lambda, \gamma \rangle<0$ for all $\gamma \in \Delta_{(\pp_u\cap \mm)}$.}
 \vspace{1ex}

Let us introduce the following groups:
First consider $\overline{M}:=Z_G(\lambda)$, with  root system $\Delta_{\overline{M}}=\{\gamma \in \Delta| \langle \gamma;\lambda \rangle=0\}$.
It is a Levi subgroup of
$$\overline{Q}:=\{g\in G|  \ \text{there exists a} \ \lim_{t\rightarrow 0} \lambda(t) g \lambda(t)^{-1} \ \text{in} \  G\},$$
with  Lie algebra
 $$\overline{\qq}=\tm \oplus \bigoplus_{\langle \alpha, \lambda \rangle\geqslant 0}\gg_\alpha. $$
The unipotent  radical of $\overline{Q}$  and the corresponding Lie algebra can be expressed by the formulae:
$$ \overline{Q}_{u}=\{g\in G|   \ \lim_{t\rightarrow 0} \lambda(t) g \lambda(t)^{-1}=e\}
 \ \ \ \ \ \overline{\qq}_u= \bigoplus_{\langle \alpha, \lambda \rangle> 0}\gg_\alpha. \  $$
In particular we have following the obvious inclusions $\overline{M}\supset L$ and $\overline{\qq}_u\supset (\pp^{-}_u\cap \mm)$.

Let us fix an open $P$-invariant subset $X^\circ=P*_LZ$  of $X$ constructed in the local structure theorem \ref{lst} applied
to some effective $B$-divisor with stabilizer equal to the parabolic
subgroup $P:=P(X)$. We recall that $L_0$  acts trivially on $Z$.
Consider the following open subset of Bia\l{}ynicki-Birula cell:
$$Z_\lambda:=\{x\in X|  \  \ \lim_{t\rightarrow 0} \lambda(t) x\in Z\},$$
It is well defined since $\lambda(t)$ fixes the points of $Z$.
Let us define a map $\varphi$  by the formula:
$$\varphi: Z_\lambda \rightarrow Z \ \ \ \varphi(x)=\lim_{t\rightarrow 0} \lambda(t)x.$$

Let us show that $ Z_\lambda\subset X^\circ$. Indeed if $\lim_{t\rightarrow 0} \lambda(t) x=z\in Z$, then the orbit  of one-parameter subgroup $\lambda$ of
the point $x$ intersects
$X^\circ$ which is an open neighborhood of $z$. Thus $X^\circ$ being $\lambda$-invariant contains the whole $\lambda(t)x$.

\begin{lemma}\label{P_u_phi} For $x\in  Z_\lambda$, $q\in \overline{Q}_{u}$ and $m\in \overline{M}$
we have the following equalities $\varphi(qx)=\varphi(x)$ and  $\varphi( m x)=m\varphi(x)$.
\end{lemma}
\begin{proof} Indeed for $q\in \overline{Q}_{u}$ by definition $\lim_{t\rightarrow 0} \lambda(t) q \lambda(t)^{-1}=e$.
Thus  we get:
$$\varphi(qx)=\lim_{t\rightarrow 0} \lambda(t) q \lambda(t)^{-1}\cdot \lim_{t\rightarrow 0}  \lambda(t) x=\lim_{t\rightarrow 0} \lambda(t) x=\varphi(x),$$
$$\varphi( mx)=\lim_{t\rightarrow 0} \lambda(t) m x=m\lim_{t\rightarrow 0} \lambda(t) x=\varphi(x)$$
\end{proof}

 \begin{proposition}\label{varphi_cap open} 
   For $z\in Z$ we have  $\varphi^{-1}(z)=(P_u\cap \overline{Q}_{u}) z$ and
  $ Z_\lambda=(P_u\cap \overline{Q}_{u})Z.$
\end{proposition}
\begin{proof} As was noticed before $Z_\lambda\subset X^\circ$.
Let us write down the action of $\lambda$ on the point $x\in X^\circ$,
that we write in the form $x=p_uz$, for $p_u\in P_u$ and $z\in Z$. Since
the action of $\lambda$ is trivial  on $Z$ we get:
$$\lim_{t\rightarrow 0}  \lambda(t)(p_uz)=\lim_{t\rightarrow 0} \lambda(t) p_u\lambda(t)^{-1}z $$
    Taking into account that $\lambda(t) p_u\lambda(t)^{-1}\in P_u$ and the fact that the action of $P_u$
  is free on $X^\circ$  we get that
 $\lim_{t\rightarrow 0} \lambda(t) p_u\lambda(t)^{-1}z $
 exists iff there exists the  $\lim_{t\rightarrow 0} \lambda(t) p_u\lambda(t)^{-1}$.
This implies that $p_u\in \overline{Q}$.

Let us  prove that $p_u\in \overline{Q}_u$. Using the Levi decomposition $\overline{Q}=\overline{M} \ltimes\overline{Q}_{u}$
 we get the decomposition $P_u\cap \overline{Q}=P_u\cap \overline{M}\ltimes P_u\cap \overline{Q}_{u}$, and in particular  $p_u=mq_u$ for $m\in P_u\cap \overline{M}$
and $q_u\in P_u\cap \overline{Q}_{u}$. Hence we obtain:
$$\lim_{t\rightarrow 0}  \lambda(t)(p_uz)=m\lim_{t\rightarrow 0}  \lambda(t)(q_uz)=mz.$$

Thus the inclusion $mz\in Z$ is satisfied if and  only if $m=e$,
since  $m\in P_u$, $X^\circ\cong P_u\times Z$.
This gives the desired inclusion $p_u\in \overline{Q}_{u}$ and finishes the proof.
\end{proof}

Consider the $\overline{Q}_{u}$-orbits of the points from $Z$.
We shall prove that it is contained in the open subset $X^\circ$.

\begin{lemma}\label{P_u_orbits} For $z\in Z$ we have $(P_u\cap \overline{Q}_{u})z=\overline{Q}_{u}z$.
\end{lemma}
\begin{proof} By Lemma \ref{P_u_phi} and  Proposition \ref{varphi_cap open} we get
 $\overline{Q}_{u}z\subseteq \varphi^{-1}(z)=(P_u\cap \overline{Q}_{u})z.$ That implies our lemma.
\end{proof}

{\bf Step 2.} Let us define the group $\overline{U}$. Consider the group $\overline{U}_{\overline{M}}=U\cap \overline{M}$.
 Let us define the required group as
 $$
 \overline{U}=\overline{U}_{\overline{M}}\ltimes \overline{Q}_{u}\subset \overline{Q}.
 $$
 It is a maximal unipotent subgroup of $G$ being the preimage of a maximal unipotent
 subgroup in $\overline{M}$ under the morphism $\overline{Q}\rightarrow
 \overline{Q}/\overline{Q}_u\cong \overline{M}$.

 Consider the family  of orbits $\overline{U}z$ for $z\in Z$.
From \ref{P_u_orbits} we obtain that $\overline{U}z=\overline{U}_{\overline{M}}(\overline{Q}_{u}z)=\overline{U}_{\overline{M}}(\overline{Q}_{u}\cap P_u)z\subset P_uz$.
This implies in particular that the orbits $\overline{U}z$ are contained in $X^\circ$ and that $\overline{U}z_1\neq \overline{U}z_2$ for $z_1\neq z_2$
(since $P_uz_1\cap P_uz_2=\emptyset$).

\begin{lemma} The  orbit $\overline{U}z$ for $z\in Z$  is stable under  the group \
 $$\overline{S}:=(L_0\ltimes(\overline{M}\cap P_u))\ltimes \overline{Q}_{u}\subset \overline{Q}.$$
\end{lemma}
\begin{proof}
  We need  to prove that $L_0$ normalizes $\overline{U}z$.
 But this follows  from
 $$L_0\overline{U}z=L_0(\overline{M}\cap P_u) \overline{Q}_{u}z=(\overline{M}\cap P_u) \overline{Q}_{u}L_0z=\overline{U}z,$$
 where we have used that $L_0\subset G_z$.
\end{proof}

Now we are ready to prove part $(ii)$ of the theorem.
We shall give two arguments one is based on degeneration to horospherical variety the other is based on the calculation of the
 image of the moment map\footnote{It is generalization of F.Knop's proof \cite{asymp}
for nondegenerate varieties}. Let us recall the definition of conormal bundle to the foliation of $\overline{U}$-orbits:
$$
\mathcal N^*X=\{(x,\xi) \in T^*X| \ x\in \overline{U}Z, \ \ \  \langle \overline{\uu}x,\xi \rangle=0\}
$$

\vspace{1ex}

{\bf Argument 1:} By \cite{Popov_contr}  (see also \cite{weylgr}) we know that every $G$-variety
 $X$ admits a degeneration to a horospherical variety.

 \begin{proposition} \label{ContractionHor} For a $G$-variety  $X$ there exists a $G\times \Bbb K^\times$-variety $\mathcal{X}$
and a surjective $G$-invariant morphism $\tau:\mathcal {X}\rightarrow \Bbb A^1$ that is
equivariant  with respect to the action $\Bbb K^\times:\Bbb A^1$,  such that
\begin{itemize}
\item[](i) For $t\neq 0$ the fiber ${X}_t:=\tau^{-1}(t)$ is isomorphic to $X$. The fiber
${X}_0$ is a smooth horospherical variety, shrinking  ${X}$
we may assume that ${X}_0\cong G/P^{-}_0\times C$ (for some variety $C$).

\item[](ii) The morphism $\tau$ is equidimensional and flat. By shrinking $X$ and $\mathcal{X}$, we can assume that $\mathcal{X}$ and $\tau$ are smooth.

\item[](iii) For the fibers of $\tau$ we have $P({X}_t)=P$, the group $L_0$ is independent of $t$ and in particular $\aa_{{X}_t}=\aa$.
\end{itemize}

\end{proposition}

Let us choose a $B$-invariant divisor $D\subset X$
with  stabilizer $P:=P(X)$.
We  extend this divisor to a $B\times \Bbb K^\times$--divisor $\MD$ on the $G\times \Bbb K^\times$--variety
$\MX$ in the following way.
  Using the isomorphism $\MX \setminus X_0\cong X\times \Bbb K^\times$
we  extend $D$ to  a $B\times \Bbb K^\times$--divisor on $\MX\setminus {X}_0$.
 We are finished by setting  $\MD$ to be the closure of this  divisor.

Having  constructed the $B\times \Bbb K^\times$--divisor $\MD$,
  we see, that  the $P[D]$-regularity condition for the weight $\chi_D$ of the section $\sigma_D$ is the same
as the $P[\MD]$-regularity condition of the weight $\chi_{\MD}$ of the section $\sigma_{\MD}$.  So we can apply the  local structure theorem
to get the following statement.

\begin{stat}  Consider  the map $\psi_{\MD}:\MX\rightarrow \gg^*$ constructed in the local structure theorem.
 Then $\Im \psi_{\MD}=\chi_D+\pp_u$. 
 Defining $\MZ:=\psi_{\MD}^{-1}(\chi_D)$, we have
 $$\MX\setminus \MD\cong P*_L \MZ. $$
 Let $X$ be an affine variety and $D=(f)$ for some  $f\in \Bbb K[X]^{(B)}$ or $X$ be a projective variety and $D$  be a
 pullback of some $B$-invariant hyperplane section, then $\MZ\cong Z\times \Bbb A^1$.
\end{stat}
\begin{proof}  Let $X$ be an affine variety. We are left to prove that $\MZ\cong Z\times \Bbb A^1$.   
Let us recall  that $\MX/\!\!/U\cong X/\!\!/U\times \Bbb A^1$ (\cite[Prop. 11]{Popov_contr}). Then $\MD$ is the divisor of zeroes of $F\in \Bbb K[\MX]^{B}$,
that is a pullback of $f$ under projection $\MX/\!\!/U\rightarrow X/\!\!/U$. 
   The desired equality follows from:
 $$\Bbb K[\MZ]=\Bbb K[\MX\setminus \MD]^U=K[\MX]^U_F=K[X]^U_f\otimes \Bbb K[t]=K[Z]\otimes \Bbb K[t],$$
where $K[\MX]^U_F,K[X]^U_f$ are  the localizations with respect to $F$ and $f$ correspondingly. 

Taking an affine cone $\widehat{X}$ over $X$ we reduce a projective case  to  an affine (we recall from \cite{weylgr} that for projective $X$, $\MX$ is constructed by 
applying a standard construction of degeneration \cite{Popov_contr} to $\widehat{X}$ and taking a quotient by dilatations).
\end{proof}

The family of orbits constructed  in the second step can be extended to the whole variety $\MX$.
 Since $P(X)=P({X}_0)$ and $\aa_X=\aa_{{X}_0}$ we can
 choose the same $\lambda$ as in  Step 1 for the variety  $\MX$ and all its fibers $X_t$.
 We define a Bia\l{}ynicki-Birula cell for $\MX$:
 $${ \mathcal Z}_\lambda=\{x\in \MX|  \ \exists \ \lim_{t\rightarrow 0} \lambda(t) x\in \MZ\}.$$
 Applying  Proposition \ref{varphi_cap open} we get
$$ {\mathcal Z}_\lambda=(P_u\cap \overline{Q}_{u}) \times \MZ. $$  

Consider the conormal bundle $\mathcal N X$ to the constructed foliation of $\overline{U}$-orbits.
It fits into the following family of the conormal bundles to the foliations of $\overline{U}$-orbits in the fibers of $\tau:\MX\rightarrow \Bbb A^1$:
 $$\mathcal N^*\MX=\{(x,t,\xi) \in T^*\MX/\tau^*(T^*\Bbb A^1)| \ x\in \overline{U}\MZ_t,  \ \  \ \langle \overline{\uu}x,\xi \rangle=0\}
 $$
We note that the restriction of $T^*\MX/\tau^*(T^*\Bbb A^1)$   to $X_t$ is isomorphic to $T^*X_t$.
From  Proposition \ref{horospherical} we know that the map $G\times \mathcal N^*X_0\rightarrow T^*X_0$ is dominant.
This implies (by application of Lemma \ref{submersion}) that $G\times \mathcal N^*X_t\rightarrow T^*X_t$ is dominant for general $t$, which proves our claim.

{\bf Argument 2:}
The second proof is based on the study of the image of the conormal bundle $\mathcal N^*X$
under the moment map $\mu_X$. Since $\overline{S}$  normalizes the orbits $\overline{U}z$ (for $z\in Z$), we get
$$
\mathcal N^*X=\{(x,\xi) \in T^*X| \ x\in \overline{U}Z, \ \ \  \langle \overline{\ss} x,\xi \rangle=0\}
$$
$$\mu_X(\mathcal N^*X)\subset \overline{\ss}^\bot=\aa+(\pp_u\cap \overline{\mm}) +\overline{\qq}_u\supset \aa+(\pp^{-}_u\cap \mm).$$

\begin{remark} By the construction of $\overline{\qq}$ we have the inclusion
 $\overline{\qq}_u\supset (\pp^{-}_u\cap \mm)$ and the equality  $\aa+(\pp^{-}_u\cap \mm)=\aa+ (\pp_u\cap \overline{\mm} +\overline{\qq}_u)\cap \mm$.
\end{remark}

Let us denote $\overline{P}=N_G(\overline{S})$. We have $\overline{\pp}_u=\pp_u\cap \overline{\mm} +\overline{\qq}_u$.
The next proposition gives us information about the image  $\mu_X(\mathcal N_zX)$.

\begin{proposition}\label{image_mu} Let $\mathcal N_zX$ be the fiber of $\mathcal NX$ over $z\in Z$. Consider the $T$-equivariant  projection
of $\aa+\overline{\pp}_u$ to the subspace $ \aa+\overline{\qq}_u\cap\pp^-_u$  with the fibers  parallel to the subspace $\overline{\qq}\cap \pp_u$. Then the image of ${\mu_X(\mathcal N_zX)}$ under this projection is equal to $\aa+\overline{\qq}_u\cap\pp^-_u$.
\end{proposition}
\begin{proof} Consider the $T$-stable decomposition
$
\gg=\aa + (\overline{\qq}_u\cap\pp^-_u + \overline{\qq}^{\ -}_u\cap\pp_u)
 + \gg_0,
$ where $\gg_0$ is orthogonal to the
other direct summands. The restriction of the pairing to $\overline{\qq}_u\cap\pp^-_u + \overline{\qq}^{\ -}_u\cap\pp_u$  is non-degenerate
and the subspaces $\overline{\qq}_u\cap\pp^-_u$, $\overline{\qq}^{\ -}_u\cap\pp_u$ are isotropic.
Moreover $\overline{\qq}\cap \pp_u\subset  \gg_0$, and the elements of $\aa+(\overline{\qq}_u\cap\pp^-_u)$
are identified with the linear functions on  $\aa+(\overline{\qq}^{\ -}_u\cap\pp_u)$.
From the inclusion $$T_zX\supset \overline{\uu}z\oplus (\aa + \overline{\qq}^{\ -}_u\cap\pp_u)z
=(\overline{\qq}\cap \pp_u)z\oplus (\aa+ \overline{\qq}^{\ -}_u\cap\pp_u)z,$$
we see that any linear function $\eta$ on $\aa+(\overline{\qq}^{\ -}_u\cap\pp_u)$ can be lifted to an element $\xi\in T_z^*X$ that is zero on
$(\overline{\qq}\cap \pp_u)z=\overline{\uu}z $. We found $\xi \in \mathcal N^*_zX$ such that the projection of $\mu_X(\xi)$
to $\aa+\overline{\qq}_u\cap\pp^-_u$ is equal to $\eta$. This proves the proposition.
\end{proof}

\begin{remark}\label{image_mu_cor} Consider  some point $z\in Z$. The one-parameter subgroup $\lambda$ acts on $T_zX$
 since $z$ is fixed by $\lambda$. In the proof of Proposition  \ref{image_mu} we constructed  a subspace
  $V_z\subset \N^*_zX$ such that $\mu_X(V_z)$
maps   isomorphically to $\aa+\overline{\qq}_u\cap \pp^-_u$ under the $T$-equivariant projection.
Let us notice that the decomposition $$T_zX=(\overline{\qq}\cap \pp_u)z\oplus (\aa+ \overline{\qq}^{\ -}_u\cap\pp_u)z\oplus R$$
can be taken $\lambda$-equivariant this implies that $V_z$ can be chosen $\lambda$-invariant.
\end{remark}

\begin{proposition}\label{image_mu2}  We have the equality  $\overline{\mu_X(\mathcal N^*X)}=\aa+\overline{\pp}_u=\aa+\pp_u\cap \overline{\mm}+\overline{\qq}_u$.
\end{proposition}
\begin{proof} Since $\overline{P}$ normalizes the foliation of orbits  it also normalizes $\mathcal N^*X$. Thus to prove the proposition
it is sufficient to show that $\overline{P}\mu_X(\mathcal N_zX)$ is dense in $\overline{\ss}^\bot=\aa+\overline{\pp}_u$.
We shall use the following lemma.

\begin{lemma}\label{P_u_cap P orbits} Consider the action of $P_u\cap \overline{Q}$
 on $\aa+\overline{\pp}_u$. Let $\xi \in \aa+\overline{\pp}_u$ be a general point. Then the stabilizer
of $\xi$ in $P_u\cap \overline{Q}$ is trivial. Moreover if $\xi \in \aa^{pr}+\pp_u\cap\overline{\qq} $
 then $(P_u\cap \overline{Q})\xi=\xi+\pp_u\cap \overline{\qq}$.
\end{lemma}
\begin{proof} The first claim follows from the second one.
 Let us take $\xi\in \aa^{pr}$. We have the inclusion $(P_u\cap \overline{Q})\xi\subset \xi+\pp_u\cap \overline{\qq}$ (it follows from the equality $P_u\cap \overline{Q}=\exp(\pp_u\cap \overline{\qq})$ and the inclusion $[\xi,\pp_u\cap \overline{\qq}]\subset\pp_u\cap \overline{\qq}$ ).
From $\pp_u\cap \zz_\gg(\xi)=(\pp_u\cap \mm)\subset \overline{\qq}^{\, -}_u$ we get that the stabilizer of $\xi$ in $P_u\cap \overline{Q}$ is trivial and we have the equality $[\xi,\pp_u\cap \overline{\qq}]=\pp_u\cap \overline{\qq}$.
This implies that the tangent space in $\xi$ to the orbit  $(P_u\cap \overline{Q})\xi$ coincides with $\pp_u\cap \overline{\qq}$. Thus $(P_u\cap \overline{Q})\xi$
is dense in $\xi+\pp_u\cap \overline{\qq}$. Since any orbit of a unipotent group in an affine variety is closed we get
$(P_u\cap \overline{Q})\xi=\xi+\pp_u\cap \overline{\qq}$.
\end{proof}

By Proposition \ref{image_mu} we get that there exists   $\xi=\xi_0+
\xi_+\in \mu_X(\mathcal N^*_zX)$ where $\xi_0\in \aa^{pr}$ and $\xi_+\in \pp_u\cap \overline{\qq}$.
By Lemma \ref{P_u_cap P orbits} we get that  the stabilizer of $\xi$ in $\pp_u\cap \overline{\qq}$ is trivial
and $[\pp_u\cap\overline{\qq} ,\xi]=\pp_u\cap \overline{\qq}$. Since the projection of $\mu_X(\mathcal N_zX)$
to the subspace $\aa+\pp^-_u\cap\overline{\qq}_u$ is surjective we get that $\mu_X(\mathcal N^*_zX)+\pp_u\cap \overline{\qq}=\overline{\ss}^\bot$.
Calculating the differential  of the map $(P_u\cap \overline{Q})\times \mu_X(\mathcal N^*_zX)\rightarrow (P_u\cap \overline{Q})\mu_X(\mathcal N^*_zX)$
in the point $\xi$ we get $$(\pp_u\cap\overline{\qq} )\times \mu_X(\mathcal N^*_zX)\rightarrow [\pp_u\cap \overline{\qq},\xi]+\mu_X(\mathcal N^*_zX)
= \overline{\ss}^\bot.
$$
This implies that $(P_u\cap \overline{Q})\mu_X(\mathcal N^*_zX)$ is dense in $\overline{\ss}^\bot$ and proves the proposition.
\end{proof}
\vspace{2ex}

\begin{remark}
The Proposition \ref{image_mu2} can  also be proved in two different ways: via the construction
of a family of linear subspaces that tends to $\aa+\overline{\qq}_u\cap\pp_u^-$ or via degeneration
 to a horospherical variety.
\end{remark}

\begin{proof}[A second proof of Proposition \ref{image_mu2}]
Assume for a moment that we have constructed a one-parameter family of linear subspaces $V_t\subset \overline{\ss}^\bot$
such that  $V_t\subset{\mu_X(\mathcal N^*_{z_t}X)}$ for some $z_t\in Z$,  for any $t\neq 0$, and $V_0=\aa+\pp_u^-\cap\overline{\qq}_u$.  By Lemma \ref{orbit_in_IM_phi} 
${\overline{P}_uV_0}$ which contains ${\overline{P}_u(\aa+(\pp^{-}_u\cap \mm))}$ is dense in $\aa+\overline{\pp}_u=\overline{\ss}^\bot$. Applying
Lemma \ref{submersion} to the map $f_t:\overline{P_u}\times V_t\rightarrow \overline{\ss}^\bot$ and using  that $f_0$ is dominant, we get that $f_t$
 is dominant for almost all $t$.
The fact that ${\overline{P}_uV_t}$ is dense in $\overline{\ss}^\bot$ and the following chain of inclusions  proves the proposition:
$$
{\mu_X(\mathcal N^*X)}\supset {\overline{P}_u\mu_X(\mathcal N^*_{z_t}X)}\supset {\overline{P}_uV_t}.
$$
To construct the desired family of subspaces let us fix a strictly dominant one-parameter subgroup $\lambda_0:\Bbb K^*\rightarrow T$.
By Proposition  \ref{image_mu} the projection of ${\mu_X(\mathcal N^*_zX)}$ for $z\in Z$ to
$\aa+\pp^-_u\cap\overline{\qq}_u$ is surjective, thus we can fix a subspace $V_1\subset{\mu_X(\mathcal N^*_zX)}$
 that maps isomorphically to $\aa+\pp^-_u\cap\overline{\qq}_u$.


For the desired family let us take the closure of the family $V_t:=\lambda_0(t)V_1\subset {\mu_X(\mathcal N^*_{\lambda_0(t)z}X)}$
in the Grassmannian of $\dim V_1$ vector subspaces of $\gg$. 
By  next lemma (that is the characterization of the open Schubert cell as the Bia\l{}ynicki-Birula cell) we get
  that in this Grassmannian  we have the following limit $\langle V_0\rangle:=\lim_{t\rightarrow 0} \langle V_t\rangle=\langle\aa+\pp^-_u\cap\overline{\qq}_u\rangle$.

\begin{lemma}\cite[Prop. 8.5.1]{SPR} Let $\lambda(t)$ be a one-dimensional torus acting on a linear space $V$. Let $V_{\leqslant 0}$
be the sum of components of $V$ with nonpositive $\lambda$-weights, and $\pi_{\leqslant 0}$ be the corresponding
$\lambda$-equivariant projection. Consider the subspace  $W\subset V$ that maps isomorphically to
$V_{\leqslant 0}$ under the projection $\pi_{\leqslant 0}$. Then we have the following limit in the Grassmannian:
$$\lim_{t\rightarrow 0} \lambda(t)\langle W \rangle=\langle V_{\leqslant 0}\rangle.$$
\end{lemma}
\end{proof}

\begin{proof}[Third proof of Proposition \ref{image_mu2}] Let us include $\mathcal{N}^*X$ into the family
$\mathcal{N}^*\MX$ as we have done in Argument 1.
 We also know that $\mu_{{X}_t}(\mathcal{N}^*{X}_t)\subset \aa+\overline{\pp}_u$.
We can do explicit calculations  of the  moment map for horospherical variety ${X}_0\cong G/P_0^-\times C$.
Indeed the image $\mu_X(\mathcal{N}^*_{x_0}{X}_0)$ for $x_0= eP_0^-$ is identified with
 $$(\overline{\uu}+\pp_0^-)^\bot= \overline{\bb}\cap (\aa+\pp_u^-)=\aa+\pp^-_u\cap\overline{\qq}_u\supset \aa+(\pp^{-}_u\cap \mm)=\aa+(\overline{\pp}_u\cap \mm).$$
   The following
 inclusions for a horospherical variety ${X}_0$:
 $${\mu_X(\mathcal{N}^*{X}_0)} \supset{\overline{P}_u\mu_X(\mathcal{N}^*_{x_0}{X}_0)}
 \supset{\overline{P}_u(\aa+(\overline{\pp}_u\cap \mm))}$$
and the density of $\overline{P}_u(\aa+(\overline{\pp}_u\cap \mm))$ in $\aa+\overline{\pp}_u$ (see Lemma \ref{orbit_in_IM_phi})
imply that $\overline{\mu_X(\mathcal{N}^*{X}_0)}=\aa+\overline{\pp}_u$.

  Applying Lemma \ref{submersion} to the triple of varieties $(\mathcal{N}^*\MX, (\aa+\overline{\pp}_u)\times \Bbb A^1,\Bbb A^1)$ and the map
$\mu_{\MX}=(\mu_{{X}_t},t):\mathcal{N}^*{X}_t\rightarrow (\aa+\overline{\pp}_u)\times \Bbb A^1$ 
 and using  the equality
$\overline{\mu_X(\mathcal{N}^*X_0)}=\aa+\overline{\pp}_u$ we get that for a general $t\in \Bbb A^1$ we have $\overline{\mu_X(\mathcal{N}^*X_t)}=\aa+\overline{\pp}_u$. This proves  the proposition.
\end{proof}

  We shall be finished after proving next proposition
\begin{proposition}  The variety 
${\overline{P}_u^{\, -}\mathcal N^*X}$ is dense in $T^*X$.
\end{proposition}
\begin{proof} By considering the differential of the map
$\overline{P}_u^{\, -}\times\mathcal N^*X\rightarrow \overline{P}_u^{\, -}\mathcal N^*X$
we see that it is sufficient to prove that for some  $\alpha\in \mathcal N^*X$ 
the tangent space to $T^*X$ in the point $\alpha$ is equal to the sum of $\overline{\pp}^{\, -}_u\alpha$
and the tangent space to $\mathcal N^*X$.
Let us take $\alpha$ such that $\xi=\mu_X(\alpha)\in \aa+\overline{\pp}_u$ is a sufficiently general point.
 Proposition \ref{nilp_orbits} implies that that $d\mu_X$ maps isomorphically $\overline{\pp}^{\, -}_u\alpha$
 onto $\overline{\pp}^{\, -}_u\xi\cong \overline{\pp}^{\, -}_u$ and that $\overline{\pp}^{\, -}_u\xi$ is transversal to
$\aa+\overline{\pp}_u=\overline{\mu_X(\mathcal N^*X)}$ in the point $\xi$.
The intersection of the tangent space to  $\mu_X(\mathcal N^*X)$ in the point $\xi$ with the subspace  $d\mu_X(\overline{\pp}^{\, -}_u\alpha)$ is equal to $(\aa+\overline{\pp}_u)\cap\overline{\pp}_u^{\, -}\xi=0$.
This implies that $\overline{\pp}^{\, -}_u\alpha$ is transversal to $\mathcal N^*X$ in $\alpha$.
The  transversality of $\overline{\pp}^{\, -}_u\alpha$ and $\mathcal N^*X$ combined with the  equality
$\codim_{T^*X} \mathcal N^*X=\dim P_u=\dim \overline{P}_u$
 implies our claim.
\end{proof}
\end{proof}

\begin{corollary} The closure of the image of the moment map is equal to $\overline{\mu_X(T^*X)}=G(\aa+\overline{\pp}_u)=G(\aa+\pp^-_u)$.
\end{corollary}
\begin{proof} We have $\overline{\mu_X(T^*X)}=\overline{G\mu_X(\mathcal N^*X)}=G(\aa+\overline{\pp}_u)$.
The equality $G(\aa+\overline{\pp}_u)=G(\aa+\pp^-_u)$  follows from the Lemma \ref{orbit_in_IM_phi} and the equality
$(\overline{\pp}_u\cap \mm)=(\pp^-_u\cap \mm)$.
\end{proof}


The following lemma states that for the considered family of $\overline {U}$-orbits parameterized by $Z$ there exists a normalizer of general position.

\begin{lemma}\label{weaknorm} For a general point $z\in Z$ the normalizer of the orbit $\overline{U}z$ is equal to some fixed group $\widetilde{S}$
normalized by $T$.
\end{lemma}
\begin{proof}  For $z\in Z$ let $\widetilde{S}$ be  the normalizer of  $\overline{U}z$. Since $N_G(\widetilde{S})$ is a parabolic subgroup of $G$ it also contains $N_G(\overline{U})$
which is a unique Borel subgroup containing $\overline{U}$;  thus $T\subset N_G(\widetilde{S})$. Since $Z\cap \overline{U}z=z$  we have $N_G(\widetilde{S})\cap T=T_z=L_0\cap T$.
 The number of horospherical subgroups normalized by $T$ and such that $\widetilde{S}\cap T=T_0$
is finite, thus for a general point  of $Z$  the   normalizer  of $\overline{U}z$ is equal to some
 fixed horospherical subgroup $\widetilde{S}$.
\end{proof}

Let us shrink $Z$ in such a way that for  any $z\in Z$ the  normalizer of the orbit $\overline{U}z$ is equal to $\widetilde{S}$.
 We can
  define  the  normalizer $\widetilde{P}$  of the family of orbits $\overline{U}z$ for  $z\in Z$,
i.e. the group that consists of  $p\in \widetilde{P}$ such that for any $z_1\in Z$ there exists  $z_2\in Z$
such that $p\overline{U}z_1=\overline{U}z_2$. 

\begin{theorem}\label{stab_of_hor} The normalizer of the orbit $\overline{U}z$ for $z\in Z$  is equal to $\overline{S}$.
 The equality $g\overline{U} z=\overline{U} z'$ for some $z,z'\in Z$ holds iff $g\in \overline{P}$ (where $\overline{P}=N_G(\overline{S})$).
 The map $G*_{\overline{P}}\mathcal N^*X\rightarrow T^*X$
is generically finite.
\end{theorem}
\begin{proof}[First proof]  As it was shown $\widetilde{P}\supset\overline{P}$. Since $\widetilde{P}$  normalizes the family of orbits $\overline{U}z$
it also normalizes $\mathcal N^*X$ the conormal bundle to this foliation. So the composition map
$$G*_{\widetilde{P}}\mathcal N^*X\rightarrow  G\mathcal N^*X\subset T^*X$$
is well defined and dominant. Calculating the dimensions:
$$\dim \mathcal N^*X =\dim Z+\dim \overline{U}z+\codim_X\overline{U}z=\dim T^*X-\dim P_u$$
$$\dim\widetilde{P} \geqslant\dim \overline{P}=\dim L+\dim P_u\cap M+\dim \overline{Q}_u=\dim P$$
$$\dim G*_{\widetilde{P}}\mathcal NX=\dim G/{\widetilde{P}}+\dim \mathcal N^*X \leqslant
\dim T^*X$$
we get that the considered map can be dominant only when $\dim G*_{\widetilde{P}}\mathcal N^*X=\dim T^*X$ (where $z\in Z$). This can  happen only when  $\widetilde{P}=\overline{P}$ and in this case  $G*_{\overline{P}}\mathcal N^*X\rightarrow T^*X$
is generically finite.
\end{proof}

\begin{proof}[Second  proof] Let $\widetilde{S}$  be the normalizer of $\overline{U}z$.
Let us denote by $\widetilde{L}_0$ the Levi subgroup of $\widetilde{S}$ that contains $L_0$ and is normalized by
$T$ (by Lemma \ref{weaknorm} $T$ normalizes $\widetilde{S}$). Let $\widetilde{S}_z$ be the stabilizer  in $\widetilde{S}$ of the point  $z$, then we have
$\widetilde{S}=\overline{U}\widetilde{S}_z$. Taking the quotient of this equality
 by the unipotent radical $\widetilde{S}_u$ of $\widetilde{S}$  (we note that $\widetilde{S}_u\subset \overline {U}$) we obtain that 
 the image of the unipotent group
   $\overline{U}$ in $\widetilde{L}_0$ acts transitively on the homogeneous space $\widetilde{L}_0/\Im \widetilde{S}_z$
of the reductive group $\widetilde{L}_0$ (where we denoted by $\Im \widetilde{S}_z$ the image of $\widetilde{S}_z$ in
  $\widetilde{S}/ \widetilde{S}_u$). From the fact that a  variety homogeneous under a reductive group and  a unipotent group is a point,
  we get that the homogeneous space  $\widetilde{L}_0/\Im \widetilde{S}_z$
  is a point and $\Im \widetilde{S}_z=\widetilde{L}_0$. From  the Levi decomposition we get that  $\widetilde{S}_z$ contains
  a subgroup conjugated to $\widetilde{L}_0$ by an element of $s_u\in\widetilde{S}_u\subset \overline{U}$.
  Thus taking the point $s_uz$ instead of $z$ we can assume that $\widetilde{S}_z$ contains $\widetilde{L}_0$.
Since the action of $P_u$  is free on $X^\circ$  the intersection $\widetilde{L}_0\cap P_u\subset \widetilde{S}_z\cap P_u$
must be trivial, which implies $\widetilde{L}_0=L_0$.
\end{proof}

\section{Horospherical cotangent bundle}

In this section  we shall define a variety of degenerate horospheres $\Hor$. Our aim is to prove that the conormal bundle to the
family of degenerate horospheres maps birationally onto the cotangent bundle of  $\Hor$.
This is a generalization of a theorem proved by E.B.Vinberg
\cite[\S 5 Thm.3]{vin}.

Consider the $G$-translates of horospheres from the foliation constructed in  Theorem \ref{degenerate_horospheres}.
By Theorem \ref{stab_of_hor} we can identify this set with the variety $\Hor:=G*_{\overline{P}}Z$. Since $\dim P=\dim \overline{P}$ we have
$\dim \Hor=\dim X$.
Let us define the incidence variety:
$$\mathcal U:=\{(x,\H)\in X\times {\Hor} | x\in \H  \}.
$$

We note that a general point of $X$ is contained in some $\H\in \Hor$ (since $GZ$ is dense in $X$). Thus the projection $p_X:\mathcal U\rightarrow X$ is dominant.
The variety $\mathcal U$ can be identified  with  the subvariety $G*_{\overline{P}}\mathcal U_0$ of $G*_{\overline{P}}(X\times Z)$ (here
$\overline{P}$ is acting diagonally on $X\times Z$ via a standard action on $X$ and via the quotient $\overline{P}/\overline{S}$ on $Z$) where
$$\mathcal U_0=\{(x,z)\in X\times Z | x\in \overline{U}z  \}.
$$
Let us notice that $Z$ can be diagonally embedded in $\mathcal U_0$.

 For this incidence variety following Vinberg we can define  the skew conormal bundle $SN^*\mathcal U$, that we denote by $HT^*X$
 (for details see  \cite[Section 4]{vin}, \cite[Section 2]{tim}).
 
 The variety $HT^*X$ can be identified with the variety  of following triples $(x,\xi,\H)$:
$$
x\in \H\in \Hor,\ \ \ \xi\in T^*_xX, \ \ \ \xi=0|_{T_x\H}.
$$
From Theorem  \ref{stab_of_hor} we also get that $HT^*X$ is identified with $G*_{\overline P}\mathcal N$.

Consider the following commutative diagram:
$$
\xymatrix{T^*{X} \ar[d]&\ar[l]_{\widehat{p}_{X}}  HT^*X \ar[d]\ar[r]^{\widehat{p}_{\Hor}} &\ar[d] T^*\Hor& \\
X&\ar[l]_{{p}_{X}}\mathcal U \ar[r]^{{p}_{\Hor}} & \Hor &
}
$$

 Theorem  \ref{stab_of_hor} can be restated  in the following form:

\begin{theorem}  The morphism $HT^*X\stackrel{\widehat{p}_X}{\longrightarrow} T^*X$ is generically finite.
\end{theorem}

We are ready to prove the following generalization of a result of Vinberg \cite[Thm.3]{vin}.

\begin{theorem}  The morphism $HT^*X\stackrel{\widehat{p}_{\Hor}}{\longrightarrow} T^*\Hor$ is birational.
\end{theorem}
\begin{proof} Since the $T^*\Hor$ is a vector bundle over $\Hor$ and $HT^*X$ maps dominantly onto
 $\Hor$ it is sufficient to prove the claim of the theorem fiberwise.
 Let $\H\in \Hor$, then the fiber of $HT^*X$ over $\H$ is identified with $\mathcal N^*\H $ --- the conormal bundle to $\H\subset X$.
We shall prove that image of $\mathcal N^*\H $ under $\widehat{p}_{\Hor}$ is birationally isomorphic to $T^*_\H\Hor$.
Since all the maps are $G$-equivariant we can assume that $\H\in Z$, i.e. $\H=(P_u\cap \overline{Q})x$ for some $x\in Z$.
We notice that the $P_u\cap \overline{Q} $-action on $\mathcal N^*\H$ is free and  the fiber $\mathcal N^*_xX$
of the conormal bundle to $\H$ at some point $x\in \H$
 defines an obvious section  of this action. Without loss of generality
we can shrink $Z$ to get an open subset isomorphic to $A\times C$, so  we can chose $x\in C$.
 By  \cite[Section 4]{vin} the morphism $\widehat{p}_{\Hor}$ maps the fiber $\mathcal N^*_xX$ isomorphically  to the subspace $N^*_\H \Hor_x\subset T_\H^*\Hor$, where $\Hor_x=p^{-1}_X(x)$ is the set of horospheres containing $x$, and $N^*_\H \Hor_x$ is the fiber over $\H$ of the conormal bundle to $\Hor_x$
 in $\Hor$.
(Let us also notice that since $p_{\Hor}$ and $p_X$ are both surjective  and the dimension of the general fiber of $p_{\Hor}$ is $\dim (\pp_u \cap\overline{\qq})$
we have $\dim \Hor_x=\dim (\pp_u \cap\overline{\qq})$.)
From the above it is clear that  birationality of $\widehat{p}_{\Hor}$ will follow from
the following proposition.

\begin{proposition}\label{transver_of_mathcal_N} The $P_u\cap \overline{Q}$-action is generically free on $T^*_\H\Hor$.
 The subspace $N^*_\H \Hor_x$ intersects a general $P_u\cap \overline{Q}$-orbit from $T^*_\H\Hor$
 transversally in a single point.
\end{proposition}
\begin{remark} The action of $P_u\cap \overline{Q}$ on $T^*_\H\Hor$ is well
defined since $P_u\cap \overline{Q}$ stabilizes $\H$.
\end{remark}
\begin{proof}
We note that the tangent space $T_\H\Hor$ can be identified with $\gg/\overline{\ss}\times T_xC$, and the
 cotangent space $T^*_\H\Hor$ is isomorphic to
 $$\overline{\ss}^\bot\times T^*_xC\cong (\aa+\pp_u\cap\overline{\mm}+\overline{\qq}_u)\times T^*_xC.$$

From this description we see  that our problem is reduced to the study of $P_u\cap \overline{Q}$-orbits in $\overline{\ss}^\bot$. Now Lemma
\ref{P_u_cap P orbits} finishes the proof of the first part of the proposition.

Let us consider a subvariety $\Hor_{tr}$ of the variety of  horospheres $\Hor$ that is equal to $(P_u\cap \overline{Q}_u^{\, -})\times Z$.
 It defines the family of horospheres $\mathcal U|_{\Hor_{tr}}=p^{-1}_{\Hor}(\Hor_{tr})$  that maps isomorphically under $p_X$ to  $X_0$.
 Indeed it consists of translates by the elements of $P_u\cap \overline{Q}_u^{\, -}$ of the orbits $(P_u\cap \overline{Q})z$ for $z\in Z$,
 and these translates do not intersect pairwise (as follows from the freeness of the $P_u$-action
 and the equality $(P_u\cap \overline{Q}_u^{\, -})(P_u\cap \overline{Q})=P_u$ that is a corollary of \cite[Prop.28.7]{hum}).
  Hence this family maps isomorphically to the open set  $(P_u\cap \overline{Q}_u^{\, -})(P_u\cap \overline{Q}) Z=P_uZ=X^\circ$.
 If we consider the embedding of  $Z$ in $\mathcal U$, we get $p^{-1}_{\Hor}(\Hor_{tr})=P_uZ$. 
  \vspace{1ex}

 {\bf Claim} We can assume that for sufficiently general $x\in Z$ the subvarieties $\Hor_{tr}$ and $\Hor_x$ of the variety $\Hor$
  are transversal in the point $\H$ corresponding to the horosphere $\overline{U}x$. We also have the equality for tangent spaces in the point $\H\in \Hor$:
   $$T_\H\Hor=T_\H\Hor_x\oplus T_\H\Hor_{tr}   \eqno{(*)}.$$
 \begin{proof}
Shrinking the varieties $X,\Hor, \mathcal U$ we can assume that they are smooth $G$-varieties and the morphisms $p_X,p_\Hor$ are
submersive. Let us notice that the horospheres parameterized by $\Hor_{tr}$ do no intersect each other and cover the open subset $X^\circ$. This implies that
 $p_\Hor^{-1}(\Hor_{tr})$ maps isomorphically to $X^\circ$  under $p_X$. Since $p_X$ is submersive,
   $p_\Hor^{-1}(\Hor_{tr})$ also intersects each fiber $\Hor_x=p^{-1}_{X}(x)$ for each $x\in X^\circ$ transversally, exactly in one point.
   Since each fiber $\Hor_x$ maps immersively into $\Hor$ under $p_\Hor$, this implies the transversality of $\Hor_{tr}$ and $\Hor_x$ in $\Hor$.
The varieties  $p_\Hor^{-1}(\Hor_{tr})$ and $\Hor_x$ have complementary dimensions in $\mathcal U$ so the varieties  $\Hor_{tr}$ and $\Hor_x$
have complementary dimension in $\Hor$ that implies $(*)$.
  \end{proof}


 \begin{lemma}  $T_\H\Hor_{tr}$ is identified canonically with $\pp_u\cap \overline{\qq}_u^{\, -} \oplus T_x Z\subset \gg/\overline{\ss}\times T_xC$ and the fiber of the conormal
bundle $N^*_\H\Hor_{tr}$ is identified with $\pp_u\cap \overline{\qq}\subset \overline{\ss}^\bot\times T^*_xC$.
\end{lemma}
\begin{proof} The first assertion is trivial. Since  $T_\H\Hor_{tr}\cong \pp_u\cap \overline{\qq}_u^{\, -} \oplus T_\H Z \subset \gg/\overline{\ss}\times T_xC$,
the linear space $N^*_\H\Hor_{tr}$ is identified with the subspace  $\pp_u\cap \overline{\qq}\subset \overline{\ss}^\bot$ which consists of the linear functions on $\gg/\overline{\ss}$
 annihilated on $\pp_u\cap \overline{\qq}_u^{\, -}$.
\end{proof}

Dualizing the equality $(*)$  we get:
$$T^*_\H\Hor=N^*_\H\Hor_x\oplus N^*_\H\Hor_{tr}=N^*_\H\Hor_x\oplus (\pp_u\cap \overline{\qq}) \eqno{(**)}$$

We shall be finished after proving the next proposition:

\begin{proposition}{\label{inter_gen_QcpP}} The intersection of $N^*_\H\Hor_x$  and a general $P_u\cap\overline{Q}$-orbit from $\overline{\ss}^\bot\oplus T^*_xC$ 
consists of a single point.
\end{proposition}
\begin{proof}

Let us take a sufficiently general point $\xi \in N^*_\H\Hor_x$ (we assume that the projection to $\aa$ is sufficiently general). We shall prove that ${\rm Ad} (u)\xi \notin N^*_\H\Hor_x$ for nontrivial $u\in P_u\cap\overline{Q}$. We represent $u$ via  exponential map
$u=\exp(\eta)$, where $\eta=\sum_{\alpha\in \Delta(\pp_u\cap\overline{\qq})}c_\alpha e_\alpha\in \pp_u\cap\overline{\qq}.$
Consider the one-parameter subgroup $\lambda:\Bbb K^\times \rightarrow Z(L_0)$
 from the proof of Theorem \ref{degenerate_horospheres}.
 We recall that $\lambda$ is nonnegative on $\overline{\ss}^\bot=\aa+\pp_u\cap\overline{\mm}+\overline{\qq}_u$ and $\aa+\pp_u\cap\overline{\mm}$ is the component of $\overline{\ss}^\bot$ of zero $\lambda$-weight.
  Since $\lambda$ lies in the stabilizer of $x\in Z$ it preserves
 the subvariety $\Hor_x$; it also stabilizes the horosphere $\H=(P_u\cap\overline{Q})x$.  Consequently, $\lambda$ acts on the linear space $N^*_\H\Hor_x$.
  Let us choose a strictly dominant one-parameter subgroup $\lambda_0:\Bbb K^\times \rightarrow T$ (in particular
 $\langle\lambda_0,\alpha\rangle> 0$ for all $\alpha \in \Delta(\pp_u)$).

 Consider the set of
 $\alpha\in \Delta(\pp_u\cap\overline{\qq})$ such that $c_\alpha\neq 0$ and   the value $\ell=\langle\lambda,\alpha\rangle$ is
   the least possible. We  choose some $\gamma$ from this set with the smallest value $\langle\lambda_0,\gamma\rangle$.
  For the vectors $\xi,\eta$ by $\xi_\ell,\eta_\ell$ we denote the components of weight $\ell$.
   Let us  prove the following lemma.

   \begin{lemma}    For the vector ${\rm Ad}(u)\xi-\xi$ the component of weight $\ell$ with respect to $\lambda$
   is nonzero but its $T$-equivariant projection to  the subspace  $\aa+\pp^-_u\cap\overline{\qq}_u$ is zero.
    \end{lemma}
    \begin{proof} First assume that $\ell\neq 0$.
    Denote by $\xi_\aa$ and by $\xi_{\pp_u\cap\overline{\mm}}$ the projection of $\xi$ to $\aa$ and $\pp_u\cap\overline{\mm}$ respectively.
       The component of $\xi$ of $\lambda$-weight zero is equal to $\xi_0=\xi_\aa+\xi_{\pp_u\cap\overline{\mm}}$.

        Let us notice that
        $$
        {\rm Ad} (u)\xi=\xi+[\eta,\xi]+\frac{(\sum_{\alpha\in \Delta(\pp_u\cap\overline{\qq})}c_\alpha {\rm ad}(e_\alpha))^2}{2!}\xi+\ldots
        $$
        The  weights in the $\lambda$-weight decomposition of $\xi$  are nonnegative.
        Since  $\ell$ has minimal possible value on the root subspaces in the exponential decomposition of $u$ and application of each $e_\alpha$ increases the $\lambda$-weight by $\langle \lambda,\alpha \rangle\geqslant \ell$, then the component  of
        ${\rm Ad} (u)\xi$ of  weight $\ell$ (which we denote by $({\rm Ad} (u)\xi)_\ell$), does not contain the summands that consist of the monomials on ${\rm ad}(e_\alpha)$ applied to the components of $\xi$
         with the $\lambda$-weight $>0$ or the monomials of degree $>1$ applied to  $\xi_0$. This implies that $({\rm Ad} (u)\xi)_\ell=\xi_\ell+[\eta_\ell,\xi_0]$.
         Since $\eta_\ell\in \pp_u$ and $\xi_0\in \aa+\pp_u$ then $[\eta_\ell,\xi_0]\in \pp_u$. 
            In particular
            $[\eta_\ell,\xi_0]$  has zero projection to $\aa+\pp^-_u\cap\overline{\qq}_u$;
         this proves the second part of the lemma in the case $\ell\neq 0$. From the above we also see that the component of ${\rm Ad}(u)\xi-\xi$ of weight $\gamma$
         is equal to $[e_\gamma,\xi_\aa]$.
         Indeed, there are no other components of this $T$-weight since the application of elements ${\rm ad}(e_\alpha)$ (with the $\alpha$ such that
          $\langle \lambda_0,\alpha\rangle\geqslant \langle \lambda_0,\gamma\rangle$) to the components of $\xi$ which belong to
         $\pp_u\cap\overline{\mm}$ (that have $\lambda_0$-weight strictly bigger than zero) give rise to a component with  $\lambda_0$-weight strictly bigger than
          $\langle \lambda_0,\gamma\rangle$.
         The component $[e_\gamma,\xi_\aa]$ is nonzero since $(\pp_u\cap\overline{\qq})\cap\zz_\gg(\aa)=0$.

         If $\ell=0$ we see that  the projection of $(P_u\cap\overline{Q})\xi$ to  the component  of zero $\lambda$-weight
         is equal to  $(P_u\cap \overline{M})\xi_0=\xi_0+\pp_u\cap\overline{\mm}$
        and the stabilizer of $\xi_0$ in $P_u\cap \overline{M}$ is trivial. This implies that ${\rm Ad}(u)\xi-\xi$ has nonzero
       projection to $\pp_u\cap\overline{\mm}$ and its projection to $\aa$ is trivial.
 \end{proof}

         Assume that $\xi,{\rm Ad} (u)\xi \in N^*_\H\Hor_x$.
         Since $N^*_\H\Hor_x$ is $\lambda$-invariant  the components  $\xi_\ell$ and  $({\rm Ad} (u)\xi)_\ell$ of weight $\ell$
         with respect to $\lambda$ also belong to  $N^*_\H\Hor_x$. By the previous Lemma
         $\xi_\ell$ and  $({\rm Ad} (u)\xi)_\ell$ are different vectors with  the same projections
         to the subspace $\aa+\pp^-_u\cap\overline{\qq}_u$.
         We note that $\xi$ and $u\xi$ have the same projection to $T^*_xC$ that implies that
          $\xi_\ell$ and  $({\rm Ad} (u)\xi)_\ell$ have the same projection to the subspace $(\aa+\pp^-_u\cap\overline{\qq}_u)\oplus T^*_xC$.
         By $(**)$ the linear subspace $N^*_\H\Hor_x$ projects isomorphically to $(\aa+\pp^-_u\cap\overline{\qq}_u)\oplus T^*_xC$.
     In particular there are no  distinct vectors in $N^*_\H\Hor_x$ with the same projection to $(\aa+\pp^-_u\cap\overline{\qq}_u)\oplus T^*_xC$.
          We come to contradiction that proves the proposition.
\end{proof}


\
We have the following corollary of the proof of Proposition \ref{inter_gen_QcpP}.

\begin{proposition}\label{inters_corr} Let $\N^*_zX$ be the fiber of the conormal bundle to the foliation of degenerate horosperes
 at some point $z\in Z$ and $V_z\subseteq \N^*_zX$
be a $\lambda$-invariant subspace such that $\mu_X(V_z)$
is mapped   isomorphically to $\aa+\pp^-_u \cap \overline{\qq}_u$ under the $T$-equivariant projection to this subspace.
Then the intersection of $\mu_X(V_z)$  and a general $P_u\cap\overline{Q}$-orbit from $\overline{\ss}^\bot$
consists of a single point.
\end{proposition}
\begin{proof} The proof of   Proposition \ref{inter_gen_QcpP} goes word by word if we use that
 $\mu_X(V_z)$ is $\lambda$-invariant and  maps   isomorphically to $\aa+\pp^-_u \cap \overline{\qq}_u$
  under the $T$-equivariant projection.
\end{proof}

\end{proof}

\end{proof}

\section{The Little Weyl group.}

By  Theorem \ref{stab_of_hor} we have the generically finite map $G*_{\overline P} \mathcal N^*X \rightarrow T^*X$ which  in fact is a rational Galois covering. The aim of this section
is to prove that the Galois group of this covering is equal to the little Weyl group of
$X$ introduced by Knop in \cite{weylgr}.

Consider the following commutative diagram.

$$
\xymatrix{
G*_{\overline P} \mathcal N^*X \ar[d]\ar[r]^{\mu_{\mathcal N^*}} &  G*_{\overline P} (\aa+\overline{\pp}_u)
 \ar[d]\ar[r]^{} &\ar[d] \aa& \\
T^*X \ar[r]^{\mu_X} & G(\aa+\overline{\pp}_u) \ar[r]^{} & \tm/W &
} 
$$

Here the upper right horizontal arrow  is a rational quotient by the group $G$.   The lower right arrow  is the composition of 
the categorical quotient by $G$ and  the Chevalley
isomorphism $\gg/\!\!/G=\tm/W$. The central vertical arrow is $[g*\xi]\rightarrow g\xi$. The map $\mu_{\mathcal N^*}$
is the following $[g*\xi]\rightarrow [g*\mu_X(\xi)]$. It is well defined since $\mu_X$ is $G$-equivariant
and in particular $\overline{P}$-equivariant.

The little Weyl group can be defined as follows, after F.Knop \cite{asymp,weylgr}. Consider the fiber product $\aa\times_{\tm/W}T^*X$.
 In general it is not irreducible. There is a natural embedding of $\N^*X$ in $\aa\times_{\tm/W}T^*X$, for  $\eta \in \N^*X$
it is defined by taking the semisimple part of the Jordan decomposition for $\mu_X(\eta)$ in $\aa$ and by embedding  $\eta$ in $T^*X$.
Let us denote  by $\widehat{T^*X}$ an irreducible component  of $\aa\times_{\tm/W}T^*X$ containing the image of $\N^*X$. Define the action of the Weyl group
$N_G(\aa)/Z_G(\aa)$ on $\aa\times_{\tm/W}T^*X$  by its action on the left multiple.

\begin{definition} The maximal subgroup $W_X$ of $N_G(\aa)/Z_G(\aa)$ that preserves the irreducible
 component $\widehat{T^*X}$ is called the little Weil group of $X$.
\end{definition}

The aim of this section is to prove the following:

\begin{theorem}\label{main_cover} The map $G*_{\overline P} \mathcal N^*X\rightarrow T^*X$ is a rational Galois covering with group $W_X$.
\end{theorem}

To prove the Theorem \ref{main_cover} we  need the notion of normalized moment map
 $\widetilde{\mu}_X:T^*X\rightarrow M_X$ introduced by Knop in \cite{weylgr}.
It can be defined via taking the Stein factorization  $T^*X\rightarrow M_X \rightarrow G(\aa+\pp^-_u)$
of the moment map $\mu_X$. In other words we take for $M_X$  a  normalization of $G(\aa+\pp^-_u)$
in the field of rational functions of $T^*X$.
We remind the reader that for the horospherical variety $G/P^{-}_0$ the variety  $G*_{P^-}(\aa+\pp^-_u)$ is a $G$-birational model for $M_{G/P^-_0}$(see \cite[\S 4]{weylgr}).



The  next lemma provides different birational $G$-models  of  $M_{G/P^-_0}$.

\begin{lemma}The varieties $G*_{P^-}(\aa+\pp^-_u)$ and $G*_{\overline{P}}(\aa+\overline{\pp}_u)$ are birationally isomorphic
to $G*_M(\aa+M*_{M\cap P^-}({\pp}_u^-\cap \mm))$ as $G$-varieties.
\end{lemma}
\begin{proof}
By Proposition \ref{nilp_orbits} there exists  $\xi_n\in \pp_u^-\cap \mm$ such that $(P^-\cap M)\xi_n$ is dense in $\pp_u^-\cap \mm$.
Since $P^-_u(\aa+(\pp^-_u\cap \mm))$ is dense in  $\aa+{\pp}^-_u$ we get that
   ${P^-(\aa+\xi_n)}\supset {P^-_u((P^-\cap M)(\aa+\xi_n))}$ is also dense there.
  This implies that   $\aa+\xi_n$ intersects a general $P^-$-orbit from $\aa+\pp^-_u$ in at most  one point.
Since $Z_{P^-}(\aa+\xi_n)=Z_{M\cap P^-}(\xi_n)$ we get that the map $P^-*_{Z_{M\cap P^-}(\xi_n)}(\aa+\xi_n)\rightarrow \aa+\pp_u$ is birational thus $G*_{P^-}(\aa+\pp^-_u)$
is $G$-birational to $G*_{P^-}(P^-*_{Z_{M\cap P^-}(\xi_n)}(\aa+\xi_n))\cong G*_{Z_{M\cap P^-}(\xi_n)}(\aa+\xi_n)$ that is contained in $G*_M(\aa+M*_{M\cap P^-}(\pp^-_u\cap \mm))$ as the dense subset.

Using the equality $(\overline{P}\cap M)= (P^-\cap M)$ and repeating literally all the above arguments for the group $\overline{P}$ instead of
$P^-$, we see that $\aa+\xi_n$ also provides a section
for the action of $\overline{P}$ on $(\aa+\overline{\pp}_u)$ and we get $G$-birational isomorphism between $G*_{\overline{P}}(\aa+\overline{\pp}_u)$ and
 $G*_M(\aa+M*_{M \cap \overline P}(\overline{\pp}_u\cap \mm))$.
\end{proof}

First let us fit the conormal bundle $\mathcal N^*X$ into the family ${\mathcal N}^*\MX$  of conormal bundles in the fibers of
$\MX\rightarrow \Bbb A^1$. We have the following proposition.

\begin{proposition}\label{fibers_M_X} The general fibers of the morphisms $\mu_X:\N^*X \rightarrow \aa+\overline{\pp}_u$,  
$\mu_{\MX}:\N^*\MX \rightarrow \aa+\overline{\pp}_u$,
 $\mu_{\mathcal N^*}:G*_{\overline P} \mathcal N^*X \rightarrow  G*_{\overline P} (\aa+\overline{\pp}_u)$ and 
 $\mu_{\mathcal N^*\MX}:G*_{\overline P} \mathcal N^*\MX \rightarrow  G*_{\overline P} (\aa+\overline{\pp}_u)$ are irreducible.
\end{proposition}
\begin{proof}
We shall need the following lemma.

\begin{lemma}\label{section_irr} Let $X$  be a normal variety and $f:X\rightarrow Y$ be a dominant morphism. Assume we have a rational
section of $f$, i.e. $\sigma: Y\dashrightarrow X$, such that $f\circ \sigma=id_Y$.  Then the general fiber of $f$ is irreducible.
\end{lemma}
\begin{proof} Consider a variety $\widetilde {Y}$ that is equal
  to the normalization of $Y$ in the field  of rational functions on $X$. Then we have two morphisms
$\widetilde{f}: X\rightarrow \widetilde {Y}$, and  $\pi:\widetilde {Y} \rightarrow Y$  such that $f=\pi\circ \widetilde{f}$,
the general fiber of $\widetilde{f}$ is irreducible, and $\pi$ is finite. Then the composition $\widetilde{f}\circ\sigma$
gives a birational section of the finite morphism $\pi$, that proves that $\pi$ is birational and gives the irreducibility of the general fiber of $f$.
\end{proof}

To prove an irreducibility of the general fiber for $\mu_X|_{\mathcal N^*X}$ by  Lemma \ref{section_irr} it is sufficient to construct a rational section
$\aa+\overline{\pp}_u \dashrightarrow \N^*X$ of the morphism $\mu_X|_{\mathcal N^*X}:\mathcal N^*X\rightarrow \aa+\overline{\pp}_u$.
 Let us notice that by Proposition \ref{image_mu}
and Remark \ref{image_mu_cor} there exists a subspace $V_z\in \N^*_zX$ such that $\mu_X(V_z)$ is isomorphic to $V_z$ and $\mu_X(V_z)$ projects isomorphically to
$\aa+\pp_u^-\cap \overline{\qq}_u$ under  the $T$-equivariant projection from $\overline{\ss}^\bot$ to $\aa+\pp_u^-\cap \overline{\qq}_u$ with the fibers
parallel to the subspace $\pp_u\cap \overline{\qq}$.
By Proposition \ref{inters_corr} the map
$(P_u\cap \overline{Q})\times \mu_X(V_z)\rightarrow \aa+\overline{\pp}_u$ is birational. Since the map
$(P_u\cap \overline{Q})\times V_z \rightarrow  (P_u\cap \overline{Q}) V_z \subset\N^*X$
is an isomorphism onto its image, from the $G$-equivariance of $\mu_X$ the variety $(P_u\cap \overline{Q}) V_z$ defines a rational section of $\mu_X:\N^*X \rightarrow \aa+\overline{\pp}_u$.
The  general  fibers for 
$\mu_\MX|_{\mathcal N^*\MX}$ are irreducible since the section constructed above is a section for this morphism as well.

The irreducibility of  general  fibers for $\mu_{\mathcal N^*}$ and $\mu_{\mathcal N^*\MX}$  now follows from the
next trivial lemma:

\begin{lemma}\label{triv} Let $G\supset H$ be algebraic groups, $X,Y$ be normal $H$-varieties and $f_H:X\rightarrow Y$ be an $H$-morphism. Let $f_G:G*_HX
\rightarrow G*_HY$ be a morphism
defined by $[g*x]\mapsto [g*f_H(x)]$ then $f_G^{-1}(f_G([g*x]))=g*f_H^{-1}(f_H(x))$.
\end{lemma}

\end{proof}

Let us denote by $\Theta$  the irreducible component of the intersection of $\mu_X^{-1}(\aa^{pr}+(\pp_u^-\cap \mm))$ and $\mathcal N^*X$ 
that maps dominantly to $\aa+(\pp_u^-\cap \mm)$ (it is unique by Proposition \ref{fibers_M_X}). Consider  $\Sigma=M\Theta$, it is
a component of $\mu_X^{-1}(\aa^{pr}+M(\pp_u^-\cap \mm))$ that maps dominantly to $\aa+M(\pp_u^-\cap \mm)$  and intersects $\mathcal N^*X$. 
Since $G(\aa^{pr}+M(\pp_u^-\cap \mm))$ is dense in $\mu_X(T^*X)$ we get that $G\Sigma$ is dense in $T^*X$. 
Let $\xi\in \aa^{pr}+M(\pp_u^-\cap \mm)$, if ${\rm Ad}(g)\xi \in (\aa^{pr}+M(\pp_u^-\cap \mm))$ for some $g\in G$ from the uniqueness of the Jordan decomposition 
we get that the semisimple parts of  $\xi$ and ${\rm Ad}(g)\xi$ are conjugated by $g$ and both lie  in $\aa^{pr}$.
Thus  the set $\{g\in G\  |  \  g\Sigma\cap\Sigma \neq \emptyset \}$ is contained in  a finite union of the cosets of $M$ in $N_G(\aa)$.
By   \cite[Lemma 2]{vin} the morphism $G*_M\Sigma\rightarrow T^*X$ is a rational Galois covering with the Galois group $N_X/M$, where
$$N_X:= \{g\in N_G(\aa)\  |  \  g\overline{\Sigma}=\overline{\Sigma} \},$$
and the action of $N_X/M$ is defined by $nM\circ[g*z]=[gn^{-1}*nz]$ for $n\in N_X$.

\begin{theorem} \label{M_X_formula} The varieties $T^*X$ and $M_X$ are $G$-birationally isomorphic to   $G*_{N_X}\Sigma$ and $G*_{N_X}(\aa^{pr}+M*_{P^-\cap M}(\pp_u^-\cap \mm))$. 
The map $\Phi:M*_{P^-\cap M}\Theta \rightarrow M\Theta$ is a birational isomorphism and
the normalized moment map ${\widetilde \mu}_X:T^*X\rightarrow M_X $ is described on some open subset by
the formula ${\widetilde \mu}_X([g*\eta])=[g*{\mu_{\N^*}}(\Phi^{-1}(\eta))]$. 
\end{theorem}
\begin{proof}[Proof of the Theorem \ref{M_X_formula}]  

To prove that the morphism $\Phi$ is birational it is sufficient to show  that it is isomorphism on the open subset
$(P_u\cap M)\times \Theta$ in $M*_{P^-\cap M}\Theta$. Let us notice that the projection of $\N^*X$ to $X$
is equal to $(P_u\cap \overline{Q})Z$ and $\overline{Q}\cap (P_u\cap M)=\{e\}$.
Thus we have an isomorphism $(P_u\cap M)(P_u\cap \overline{Q})Z\cong (P_u\cap M)\times (P_u\cap \overline{Q})Z$
which implies $(P_u\cap M)\N^*X\cong (P_u\cap M)\times \N^*X$. In particular 
$(P_u\cap M)\Theta\cong (P_u\cap M)\times\Theta$.

We have  an $M$-equivariant  rational map $\mu_{\N^*}\circ\Phi^{-1}:\mu_X^{-1}(\aa^{pr}+M(\pp_u^-\cap \mm))\dashrightarrow M*_{P^-\cap M}(\aa^{pr}+(\pp_u^-\cap \mm))$,
that induces the rational map:
 $$\widetilde{\mu}_X:G*_{N_X}\Sigma \dashrightarrow G*_{N_X}(\aa^{pr}+M*_{P^-\cap M}(\pp_u^-\cap \mm)),$$
which factors the moment map $\mu_X$. To prove that $M_X$ is birational to  $G*_{N_X}(\aa^{pr}+M*_{P^-\cap M}(\pp_u^-\cap \mm))$
it remains to prove irreducibility of  general fibers for  $\widetilde{\mu}_X$.

The  irreducibility of  general fibers for $\mu_X|_{\N^*X} \rightarrow \aa+\overline{\pp}_u$ (see Proposition \ref{section_irr})
 implies the irreducibility of general fibers for the maps $\mu_X|_\Theta$,  
$M*_{P^-\cap M}\Theta \rightarrow M*_{P^-\cap M}(\aa^{pr}+(\pp_u^-\cap \mm))$
and ${\widetilde \mu}_X$.
\end{proof}

\begin{remark}\label{stab_T*X} Since  ${G\N^*X}$ is dense in $T^*X$ to study the stabilizer of a general cotangent vector $\eta_x\in T_x^*X$ 
it  suffice  to consider $\eta_x\in \N^*X$, and acting by $\overline{P}$ we may assume that $\xi=\mu_X(\eta_x)\in\aa^{pr}+ (\overline{\pp}_u\cap \mm)$.
From the description of $M_X$ we get that 
$G_{\eta_x}\subset \overline{S}\cap G_\xi$.
\end{remark}

\begin{proof}[Proof of the Theorem \ref{main_cover}] As we have seen the Galois group of the rational covering $G*_{\overline P} \mathcal N^*X\rightarrow T^*X$
is equal to $N_X/M$ and by Theorem \ref{M_X_formula}  the map $M_{G/P^-_0}\rightarrow M_X$ is the quotient by the action of $N_X/M$ which  commutes with the $G$-action.
By definition of Knop the little Weil group is the Galois group of the covering $M_{G/P^-_0}/\! \!/ G\rightarrow M_X/\! \!/ G$ (we note that $M_{G/P^-_0}/\! \!/ G \cong \aa$) which  is also equal to $N_X/M$.
\end{proof}

\end{document}